\input ./style/arxiv-general.cfg
\documentclass[aop,MSNbibl,seceqn,dvips]{arximspdf}
\makeatletter
   \@ifpackageloaded{graphicx}{}{\usepackage{graphicx}}
\makeatother
\usepackage{mathrsfs}

\doi{10.1214/13-AOP871} 
\volume{43}
\issue{3}
\pubyear{2015}
\firstpage{1045}
\lastpage{1081}

\makeatletter
\def\mid{|}
\def\xrightarrow{\longrightarrow}
\newcommand{\eqref}[1]{(\ref{#1})}
\newtheorem{proposition}{Proposition}[section]

\newtheorem{theorem}[proposition]{Theorem}

\newtheorem{lemma}[proposition]{Lemma}

\newproclaim{example}[proposition]{Example}
\newproclaim{remark}[proposition]{Remark}
\newproclaim{remarkk}{Remark}[section]
\newtheorem{lemmaa}[remarkk]{Lemma}

\renewcommand\P{\mathbb{P}}
\newcommand\E{\mathbb{E}}

\newcommand{\ind}{\mathrm{d}}
\newcommand{\inp}{\mathrm{p}}
\newcommand{\as}{\mathrm{a.s.}}

\newcommand{\map}{\mathbf{m}}
\newcommand\mean{\kappa}
\newcommand\ga{\alpha}
\newcommand\gb{\beta}
\newcommand\gD{\Delta}
\newcommand\gG{\Gamma}
\newcommand\gs{\sigma}
\newcommand\eps{\varepsilon}

\newcommand{\tend}{\longrightarrow}
\newcommand\pto{\stackrel{\mathrm{p}}{\tend}}

\newcommand\bbR{\mathbb R}

\newcommand\xxi{\xi}
\newcommand\xS{S}
\newcommand\So{S\oo}
\newcommand\oo{^{(0)}}
\newcommand\Ge{\operatorname{Ge}}
\newcommand\NBi{\operatorname{NegBin}}
\newcommand\hxi{\xi}
\newcommand\cG{\mathcal{G}}
\newcommand\cM{\mathcal{M}}
\newcommand\cT{\mathcal{T}}
\newcommand\ex{\mathbf{e}}
\newcommand\brmin{U}
\newcommand\hp{p}
\newcommand\LL{\bar L}
\newcommand\br{\mathbf{b}}
\newcommand\h{\mathbf{h}}
\newcommand\vv{\hat v}
\makeatother

\begin{document}
\begin{frontmatter}

\title{Scaling limits of random planar maps with a~unique~large face\thanksref{T1}}
\runtitle{Scaling limits of random planar maps}

\begin{aug}
\author[A]{\fnms{Svante} \snm{Janson}\ead[label=e1]{svante.janson@math.uu.se}}
\and
\author[B]{\fnms{Sigurdur \"Orn} \snm{Stef\'ansson}\corref{}\ead[label=e2]{rudrugis@gmail.com}}
\address[A]{Department of Mathematics\\
Uppsala University\\
PO Box 480\\
SE-751~06 Uppsala\\
Sweden\\
\printead{e1}}
\address[B]{Department of Mathematics\\
Uppsala University\\
PO Box 480\\
SE-751~06 Uppsala\\
Sweden\\
and\\
Nordita, the Nordic Institute\\
\quad for Theoretical Physics\\
KTH Royal Institute of Technology\\
\quad and Stockholm University\\
Roslagstullsbacken 23\\
SE 106 91 Stockholm\\
Sweden\\
\printead{e2}}
\thankstext{T1}{Supported in part by the Knut and Alice Wallenberg Foundation.}
\affiliation{Uppsala University, and Uppsala University and Nordita}
\runauthor{S.~Janson and S.~\"O.~Stef\'ansson}
\end{aug}

\received{\smonth{12} \syear{2012}}
\revised{\smonth{6} \syear{2013}}

%
\begin{abstract}
We study random bipartite planar maps defined by assigning nonnegative
weights to each face of a map. We prove that for certain choices of
weights a unique large face, having degree proportional to the total
number of edges in the maps, appears when the maps are large. It is
furthermore shown that as the number of edges $n$ of the planar maps
goes to infinity, the profile of distances to a marked vertex rescaled
by $n^{-1/2}$ is described by a Brownian excursion. The planar maps,
with the graph metric rescaled by $n^{-1/2}$, are then shown to
converge in distribution toward Aldous' Brownian tree in the
Gromov--Hausdorff topology. In the proofs, we rely on the Bouttier--di
Francesco--Guitter bijection between maps and labeled trees and recent
results on simply generated trees where a unique vertex of a high
degree appears when the trees are large.
\end{abstract}

%
\begin{keyword}[class=AMS]
\kwd[Primary ]{05C80}
\kwd[; secondary ]{05C05}
\kwd{60F17}
\kwd{60J80}
\end{keyword}

\begin{keyword}
\kwd{Random maps}
\kwd{planar maps}
\kwd{mobiles}
\kwd{simply generated trees}
\kwd{continuum random tree}
\kwd{Brownian tree}
\end{keyword}
%

\end{frontmatter}

\section{Introduction}\label{s:intro}

A planar map is an embedding of a finite connected graph into the
two-sphere. Two planar maps are considered to be the same if one can be
mapped to the other with an orientation-preserving homeomorphism of the
sphere. The connected components of the complement of the edges of the
graph are called faces. The degree of a vertex is the number of edges
containing it and the degree of a face is the number of edges in its
boundary where an edge is counted twice if both its sides are incident to
the face.

In recent years, there has been great progress in understanding probabilistic
aspects of large planar maps; we refer to \cite{legall:notes} for a
detailed overview. One approach has been to study the scaling limit of a
sequence of random planar maps obtained by rescaling the graph distance on
the maps appropriately with their size and taking the limit as the size goes
to infinity. This notion of convergence involves viewing the maps as
elements of the set of all compact metric spaces, up to isometries, equipped
with the Gromov--Hausdorff topology. Le Gall showed that the scaling limit
of uniform $2p$-angulations (all faces of degree $2p$) exists along a
suitable subsequence and he furthermore showed that its topology is
independent of the subsequence and proved that its Hausdorff dimension
equals 4 \cite{legall:2007}. Subsequently, Le Gall and Paulin proved that
the limit has the topology of the sphere \cite{legall:2008}. Recently,
Miermont showed that in the case of uniform quadrangulations the
choice
of subsequence is superfluous and the scaling limit in fact equals the
so-called \textit{Brownian map} up to a scale factor \cite{miermont:unique}.
Le Gall proved independently that the same holds in the case of uniform
$2p$-angulations and uniform triangulations \cite{legall:unique}.

The present work is motivated by a paper of Le Gall and Miermont
\cite{legall:2011} where the authors study random planar maps which
roughly have the property that the distribution of the degree of a typical
face is in the domain of attraction of a stable law with index $\alpha
\in
(1,2)$. The model belongs to a class of models in which Boltzmann weights
are assigned to the faces of the map as we will now describe. Let
$\mathcal{M}^\ast_n$ denote the set of rooted and pointed bipartite planar
maps having $n$ edges: the root is an oriented edge $e = (e_{-},e_+)$ and
pointed means that there is a marked vertex $\rho$ in the planar map. The
assumption of pointedness is for technical reasons. For a planar map
$\map
\in\mathcal{M}^\ast_{n}$, denote the set of faces in $\map$ by
$F(\map)$
and denote the degree of a face $f\in F(\map)$ by $\deg(f)$. Note that
the assumption that $\map$ is bipartite is equivalent to
assuming that $\deg(f)$ is even for all $f$. Let
$(q_i)_{i\geq1}$ be a sequence of nonnegative numbers and assign a
Boltzmann weight
%
\begin{equation}
\label{wm} W(\map) = \prod_{f \in F(\map)} q_{\deg(f)/2}
\end{equation}
to $\map$. The probability distribution $\mu_n$ is defined by
normalizing $W(\map)$
%
\begin{equation}
\label{mumeasure} \mu_n(\map) = W(\map)/Z_{n},
\end{equation}
where
%
\begin{equation}
\label{eq:part1} Z_n = \sum_{\map'\in\mathcal{M}^\ast_{n}} W\bigl(
\map'\bigr)
\end{equation}
is referred to as the finite volume partition function. We will always
assume that $q_k > 0$ for some $k \geq2$ to avoid the trivial case
when all
faces have degree 2. Note that for a given random element in
$\mathcal{M}^\ast_{n}$ distributed by $\mu_n$ the marked vertex
$\rho$ is
uniformly distributed. The motivation for studying these distributions is
first of all related to questions of universality, namely, there is strong
evidence that under certain integrability condition on the weights
$q_i$ the
scaling limit of the maps distributed by $\mu_n$ is the Brownian map up
to a
scale factor \cite{marckert:2007}. Furthermore, the distributions are
closely related to distributions arising in certain statistical mechanical
models on random maps as is discussed in \cite{legall:2011}.

In \cite{legall:2011}, the authors show, among other things, that in
the large planar maps under consideration there are many
``macroscopic'' faces present and that the scaling limit, if it exists,
is different from the Brownian map. The presence of these large faces
in the scaling limit can be understood by considering the labeled trees
(mobiles) obtained from the planar maps using the Bouttier--di
Francesco--Guitter (BDG) bijection \cite{bouttier:2004}; see Section~\ref{s:bdg}. For convenience, we rewrite the sequence $(q_i)_{i\geq1}$
in terms of a new sequence $(w_i)_{i\geq0}$ defined by $w_0=1$ and
%
\begin{equation}
\label{eq:defp} w_i = \pmatrix{2i-1
\cr
i-1}q_{i},\qquad  i\ge1.
\end{equation}
%
Through yet another bijection between mobiles (with labels
removed) and trees which we introduce in Section~\ref{s:anotherb}, the
random trees corresponding to the maps distributed by $\mu_n$ can be viewed
as so-called \textit{simply generated trees} with weights $w_i$
assigned to
vertices of outdegree $i$. The choice of weights $(q_i)_{i\geq1}$ in
\cite{legall:2011} corresponds to choosing the weights $(w_i)_{i\geq
0}$ as
an offspring distribution of a critical Galton--Watson tree in the
domain of attraction of a stable law of index $\alpha\in(1,2)$. In
this case, the random trees converge, when
scaled appropriately, to the so-called stable tree with index $\alpha
$. It
follows from properties
of the BDG bijection that the large faces in the planar maps correspond
to individuals in the stable tree which have a macroscopic number of
offspring, that is, vertices of large degree.

It was originally noted in \cite{bialas:1996} and recently developed
further in \cite{janson:2012,janson:2011,jonsson:2011,kortchemski:2012}
that there exists a phase of simply generated trees where a unique
vertex with a degree proportional to the size of the tree appears as
the trees get large. This phenomenon has been referred to as
condensation. The purpose of this paper is to study the scaling limit
of planar maps corresponding to the condensation phase of the simply
generated trees. The large vertex in the trees will produce a large
face in the planar maps in analogy with the situation in \cite
{legall:2011}. The weights which we consider are chosen as explained
below. Define the generating function
%
\begin{equation}
\label{generating} g(x) = \sum_{i=0}^\infty
w_i x^i
\end{equation}
and denote its radius of convergence by $R$.\vspace*{1pt}
For $R>0$, define $\mean= \lim_{t\nearrow R} \frac{t g'(t)}{g(t)}$ and
for $R = 0$ let $\mean=0$. We will be interested in the following two
cases, (\ref{cond:A1}) and (\ref{cond:A2}), which are known to be the
only cases giving rise to condensation in the corresponding simply
generated trees (see, e.g., \cite{janson:2012}):
{\renewcommand{\theequation}{C\arabic{equation}}
\setcounter{equation}{0}
\begin{eqnarray}
\label{cond:A1} 
&0< R <\infty\quad\mbox{and}\quad \mean<1.&
\\
\label{cond:A2} 
&R = 0.&
\end{eqnarray}}
\hspace*{-3pt}In practice, we will consider the special case of (\ref{cond:A1}) when the
weights furthermore obey
\setcounter{equation}{5}
\begin{equation}
\label{A1} w_i = L(i)i^{-\beta}
\end{equation}
for some $\beta> 2$ and some slowly varying function $L$ and the
special case of (\ref{cond:A2}) when the weights furthermore obey
%
\begin{equation}
\label{A2} w_n = (n!)^\alpha
\end{equation}
with $\alpha> 0$.
[By \eqref{eq:defp} and Stirling's formula, \eqref{A1} is equivalent to
$q_i= L'(i) 4^{-i} i^{1/2-\gb}$ for another slowly varying function $L'$;
the exponential factor $4^i$ does not matter when we fix the number of edges
in the map, so we might as well take $q_i=L'(i) i^{1/2-\gb}$. However, we
will in the sequel use $w_i$ rather than $q_i$.]

We now introduce some formalism needed to state the results of the
paper. Let $\mathbb{M}^\ast$ be the set of all pointed compact metric
spaces viewed up to isometries, equipped with the pointed
Gromov--Hausdorff metric $d_{\mathrm{GH}}$ \cite{gromov:1999}. Let $\mathbf{e}$
be a standard Brownian excursion on $[0,1]$ and denote by $(\mathcal
{T_\mathbf{e}},\delta_\mathbf{e})$ Aldous' continuum random tree coded
by $\mathbf{e}$. Recall that $\mathcal{T}_\mathbf{e} = [0,1]/{\{
\delta
_\mathbf{e}=0\}}$ where
%
\begin{equation}
\delta_\mathbf{e}(s,t) = \mathbf{e}(s)+\mathbf{e}(t)-2\inf
_{s\wedge
t<u<s\vee t}\mathbf{e}(u)
\end{equation}
and by abuse of notation $\delta_\mathbf{e}$ is the induced distance on
the quotient; see, for example, \cite{aldous:1993,legall:notes}.
From here on, we will denote a random element in $\mathcal{M}^\ast_{n}$
distributed by $\mu_n$ by $(M_n,\rho)$; sometimes simplified to $M_n$.
The graph distance in $M_n$ will be denoted by $d_n$.

The main results of the paper are the following.
In Theorem~\ref{th:Dinv}, we prove that for the weights (\ref{A1}) and
(\ref{A2}), the limit as $n\rightarrow\infty$ of the profile of distances
in $M_n$ to the marked vertex $\rho$, rescaled by $(2(1-\mean)n)^{-1/2}$,
is described by a standard Brownian excursion; see Section~\ref{s:label} for
definitions and a precise statement.
Second, we prove the following
theorem, which describes the limit of all distances (not just to the root).

\begin{theorem} \label{th:main}
For the weights \eqref{A1} and \eqref{A2}, the random planar maps
$((M_n,\rho),(2(1-\mean)n)^{-1/2}d_n)$ distributed by $\mu_n$ and
viewed as
elements of $\mathbb{M}^\ast$ converge in distribution to
$((\mathcal{T}_\mathbf{e},\rho^\ast),\delta_\mathbf{e})$, where given
$\mathcal{T}_\mathbf{e}$, $\rho^*$ is a marked vertex chosen
uniformly at
random from $\mathcal{T}_\mathbf{e}$.
\end{theorem}

Note that the root edge in $\cM_n$ is forgotten when we regard the
maps as
elements of $\mathbb{M}^\ast$. We can reroot the random tree
$\mathcal{T}_\mathbf{e}$
at the
randomly chosen point $\rho^*$; this gives a new random rooted tree, which
has the same distribution as $\mathcal{T}_\mathbf{e}$,
as shown by \cite{aldous:1991}, (20),
but the point $\rho^*$ is now the root.
Hence, the result in Theorem~\ref{th:main} can also be formulated as follows.

\begin{theorem}
For the weights \eqref{A1} and \eqref{A2}, the random planar maps
in Theorem~\ref{th:main} converge in distribution in $\mathbb{M}^*$ to
$((\mathcal{T}_\mathbf{e},0),\delta_\mathbf{e})$, where
$0$ denotes the root of $\mathcal{T}_\mathbf{e}$.
\end{theorem}

Note that the limit $\mathcal{T}_\mathbf{e}$ is quite different from the
Brownian map mentioned above; it is a (random) compact tree, and thus
contractible, that is, of the same homotopy type as a point, and its Hausdorff
dimension is 2 \cite{duquesne:2005,haas:2004}. Bettinelli~\cite
{bettinelli:2011} showed a similar convergence of uniform
quadrangulations with a boundary toward Aldous' continuum random tree
when the length of the boundary grows sufficiently fast and the
distances in the quadrangulations are divided by the square root of the
length of the boundary (see also the work of Bouttier and Guitter \cite
{bouttier:2009}). In this case, the boundary grows so fast that the
faces disappear when rescaled and the boundary folds into a tree. This
is analogous to our situation where the boundary of the large face
folds into a tree; see Figure~\ref{f:conv}. Other examples of planar maps
converging to Aldous' continuum tree are stack triangulations \cite
{albenque:2008} and random dissections of polygons \cite{curien:2013}.

\begin{figure}[b]

\includegraphics{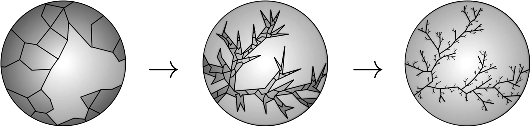}

\caption{The convergence in Theorem \protect\ref{th:main}. A face of large
degree (light gray) appears as the planar map gets larger and the
boundary of that face collapses into a tree.} \label{f:conv}
\end{figure}

The paper is organized as follows. We begin in Section~\ref{s:bdg} by
recalling the BDG bijection between planar maps and planar mobiles. In
Section~\ref{s:anotherb}, we introduce a bijection from the set of planar
trees to itself which allows us to translate results on the condensed phase
of simply generated trees to our setting. In Section~\ref{s:label}, we state
and prove Theorem~\ref{th:Dinv} which was described informally
above. Section~\ref{s:proof} is devoted to the proof of Theorem~\ref{th:main}. We end with some concluding remarks in Section~\ref{s:conc}
and \hyperref[AGW]{Appendix} containing further results on the random
Galton--Watson trees used here and their relation to the two-type
Galton--Watson trees used by Marckert and Miermont \cite{marckert:2007}.

\section{Planar mobiles and the BDG bijection} \label{s:bdg}

In this section, we define planar trees and mobiles and explain the BDG
bijection between mobiles and planar maps. We consider rooted and pointed
planar maps as is done in \cite{legall:2008} which is different from the
original case \cite{bouttier:2004} where the maps were pointed but not
rooted. (But see \cite{bouttier:2004}, Section~2.4.)

Planar trees are planar maps with a single face. It will be useful to
keep this definition in mind later in the paper but we recall a more
standard definition below and introduce some notation. The infinite
Ulam--Harris tree $T_\infty$ is the tree having a vertex set $\bigcup_{k=0}^\infty\mathbb{N}^k$, that is, the set of all finite sequences of
natural numbers, and every vertex $v = v_1 \cdots v_k$ is connected to
the corresponding vertex $v'=v_1\cdots v_{k-1}$ with an edge. In this
case, $v$ is said to be a child of $v'$ and $v'$ is said to be the
parent of $v$. The vertex belonging to $\mathbb{N}^0$ is called the
root and denoted by $r$.

A rooted planar tree $\tau$ is defined as a rooted subtree of
$T_\infty
$ having the properties that if $v = v_1\cdots v_k$ is a vertex in
$\tau
$ then $v_1\cdots v_{k-1} i$ is also a vertex in $\tau$ for every $i <
v_k$. The vertices in a planar tree have a lexicographical ordering
inherited from the lexicographical ordering of the vertices in
$T_\infty
$. This order relation will be denoted by $\leq$. Let $\Gamma_n$ be the
set of rooted planar trees with $n$ edges. We use the convention that
the root vertex is connected to an extra half-edge (not counted as an
edge) such that every vertex has degree 1${} + {}$the number of its children
(1${} + {}$its out degree). The number of edges in a planar tree $\tau$ will
be denoted by $|\tau|$.

Consider a tree $\tau_n \in\Gamma_n$ and color its vertices with two
colors, black and white, such that
the root and
vertices at even distance from the root
are white and vertices at odd distance from the root are black. Denote the
black vertex set of $\tau_n$ by $V^{\bullet}(\tau_n)$ and the white vertex
set by $V^{\circ}(\tau_n)$. If $u$ is a black vertex let $u_{0}$ be the
(white) parent of $u$ and denote by $u_{i}$ the $i$th (white) neighbor of
$u$ going clockwise around $u$ starting from $u_{0}$.

Assign integer labels $\ell_n\dvtx V^{\circ}(\tau) \rightarrow\mathbb
{Z} $ to
the white vertices of $\tau_n$ as follows: The root is labeled by 0.
If $u$
is black and has degree $k$ then
%
\begin{equation}
\label{labelrule} \ell_n(u_{j+1}) \geq\ell_n(u_{j})
- 1
\end{equation}
for all $0 \leq j \leq k$, with the convention that $u_{k} = u_{0}$.
The pair $\theta_n = (\tau_n,\ell_n)$ is called a mobile and we denote
the set of mobiles having $n$ edges by $\Theta_n$.

The set $\Theta_n \times\{-1,1\}$ is in a one to one correspondence with
the set $\mathcal{M}^\ast_{n}$ according to
(the rooted version of)
the BDG bijection
\cite{bouttier:2004,legall:2008}.
We will denote the BDG bijection by $\mathcal{F}_n\dvtx \mathcal{M}^\ast_{n} \rightarrow\Theta_n\times\{-1,1\}$ and we
give an
outline of its inverse direction below. Start with a planar mobile
$\theta_n\in\Theta_n$ and an $\varepsilon\in\{-1,1\}$. The contour
sequence of $\theta_n$ is a list $a_0,a_1,\ldots,a_{2n-1}$ of length $2n$
containing the vertices in the mobile (with repetitions allowed)
constructed as follows. The
first element is $a_0 = r$ and for each $i < 2n-1$ the element following
$a_i$ is the first child (in the lexicographical order) of $a_i$ which has
still not appeared in the sequence or if all its children have appeared it
is the parent of $a_i$. Extend this sequence to an infinite sequence by
$2n$ periodicity. The white contour sequence is defined as $c_i =
a_{2i}$, $i \geq0$. The white contour sequence can be described as
a list of the white vertices encountered in a clockwise walk around the
contour of the tree, which starts at the root. For an index $i \in
\mathbb{N,}$ define its successor as
%
\begin{equation}
\sigma(i) = \inf\bigl\{j > i \dvtx \ell_n(c_j) =
\ell_n(c_i)-1\bigr\},
\end{equation}
where the infimum of the empty set is defined as $\infty$.
Add an external vertex $\rho$ to the mobile, disconnected from all
other vertices, and write $\rho= c_\infty$. Also define the successor
of a white vertex $c_i$ as
%
\begin{equation}
\label{succ} \sigma(c_i) = c_{\sigma(i)}.
\end{equation}
A planar map is constructed from $\theta_n$ by inserting an edge
between $c_i$ and $\sigma(c_i)$ for each $0 \leq i < n$ and deleting
the edges and black vertices of the mobile. The vertex $\rho$
corresponds to the marked vertex of the planar map. The root edge in
the map is the edge between $c_0$ and $\sigma(c_0)$ and its direction
is determined by the value of $\varepsilon$, if $\varepsilon= 1$ ($\varepsilon
= -1$) the root edge points toward (away from) the root of the mobile.

%
\begin{figure}[b]

\includegraphics{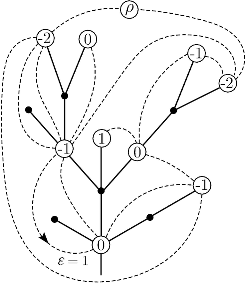}

\caption{An illustration of the BDG bijection. The edges in the mobile
are solid and the edges in the planar map are dashed.} \label{f:bdg}
\end{figure}

Thus, the white vertices in the mobile along with an additional isolated
white vertex $\rho$ correspond to the vertices in the planar map and the
black vertices in the mobile correspond to the faces in the planar map, a
face having a degree two times the degree of its corresponding black vertex;
see Figure~\ref{f:bdg} for an example. Moreover, the labels in a mobile give
information on distances to the marked vertex $\rho$ in the corresponding
planar map $\map$. Define the label of $\rho$ as $\ell_n(\rho) =
\min_{u\in
V^{\circ}(\map)} \ell_n(u)-1$. Then
%
\begin{equation}
\label{distances} d_n(v,\rho) = \ell_n(v) -
\ell_n(\rho),\qquad v\in V(\map),
\end{equation}
where by abuse of notation $\ell_n(v)$ stands for the label of the
white vertex in the mobile corresponding to the vertex $v$ in the
planar map.

The probability distribution $\mu_{n}$ on $\mathcal{M}^\ast_{n}$ is
carried to a probability distribution $\tilde{\mu}_n$ on $\Theta
_n\times\{-1,1\}$ through the BDG-bijection, that is, $\tilde{\mu
}_n(A) = \mu_{n}(\mathcal{F}_n^{-1}(A))$ for any subset $A \subseteq
\Theta_n\times\{-1,1\}$ and $\tilde{\mu}_n$ can be described as
follows: Let $\tau_n \in\Gamma_n$ and denote by $\lambda_n(\tau_n)$
the number of ways one can add labels to the white vertices of $\tau_n$
according to the above rules. One easily finds that
%
\begin{equation}
\lambda_n(\tau_n) = \prod_{v\in V^{\bullet}}
\pmatrix{2\deg (v)-1
\cr
\deg(v)-1}.
\end{equation}
This follows from counting the number of allowed label increments
around each black vertex $v$. The number of label increments around $v$
is $\deg(v)$, call them $x_1, x_2,\ldots,x_{\deg(v)}$ in say clockwise
order. The number of different configurations is then given by
%
\begin{equation}
\label{incr} \mathop{\sum_{x_1+\cdots+x_{\deg(v)} = 0 }}_{ x_i \geq-1, \forall
i}1
= \mathop{\sum_{y_1+\cdots+y_{\deg(v)}=\deg(v)}}_{ y_i \geq0,
\forall i}1 = \pmatrix{2
\deg(v) -1
\cr
\deg(v) -1},
\end{equation}
the number of compositions of $\deg(v)$ into $\deg(v)$ nonnegative parts.

A Boltzmann weight
%
\begin{equation}
\label{eq:bolmob} \tilde{W}(\tau_n) = \prod
_{v\in V^{\bullet}} \pmatrix{2\deg (v)-1
\cr
\deg (v)-1}q_{\deg(v)} =
\prod_{v\in V^{\bullet}} w_{\deg(v)}
\end{equation}
is assigned to the tree $\tau_n$ and
%
\begin{equation}
\tilde{\mu}_n\bigl(\bigl((\tau_n,\ell_n),
\varepsilon\bigr)\bigr) = \tilde{W}(\tau_n)/\bigl(\lambda_n(
\tau_n) Z_n \bigr),
\end{equation}
where $\ell_n$ is any labeling of $\tau_n$, $\varepsilon\in\{-1,1\}$ and
$Z_n = 2\sum_{\tau_n \in\Gamma_n} \tilde{W}(\tau_n)$ is the
finite volume
partition function defined in (\ref{eq:part1}). Note that given $\tau
_n$ the
labels $\ell_n$ are assigned uniformly at random from the set of all
labelings and $\varepsilon$ is chosen uniformly from $\{-1,1\}$. We will also
find it useful to study the distribution of $\tau_n$ after forgetting about
the labeling and the value of $\varepsilon$. For that purpose, we define
$\tilde{\nu}_n$ to be a probability distribution on $\Gamma_n$ given by
%
\begin{equation}
\label{tnu} \tilde{\nu}_n(\tau_n) =\sum
_{\ell_n,\eps} \tilde{\mu}_n\bigl(\bigl((
\tau_n,\ell_n),\varepsilon\bigr)\bigr) = 2\tilde{W}(
\tau_n)/Z_n.
\end{equation}

This distribution was shown by Marckert and Miermont \cite
{marckert:2007} to
be the distribution of a certain two-type Galton--Watson tree; see
\hyperref[AGW]{Appendix}.

\subsection{Distribution of labels in a fixed tree}
We provide a result which we will later need on the distribution of
the maximum absolute value of the labels in a mobile.

\begin{lemma} \label{l:labeldist}
Let $\theta_n = (\tau_n,\ell_n)\in\Theta_n$ be a mobile with $\tau
_n$ fixed
(nonrandom) and the labels $\ell_n$ chosen uniformly from the allowed
labelings of the white vertices of $\tau_n$ according to the rules
\eqref{labelrule}. For every $p>0$, there exists a constant $C(p)>0$
independent of $\tau_n$ such that
%
\begin{equation}
\label{labeldist} \mathbb{E} \Bigl(\sup_{v\in V^\circ(\tau_n)}\bigl|
\ell_n(v)\bigr|^p \Bigr) \leq C(p) n^{p/2}.
\end{equation}
\end{lemma}

To prove this lemma, we relate the labels of $\tau_n$ to a random walk
indexed by the white vertices in $\tau_n$. We start by proving the
result for $p>2$ and the general case follows by Jensen's inequality.
In the following, we will let $C_1,C_2,\ldots$ be constants which do
not depend on the tree $\tau_n$ but may depend on other quantities
which will then be explicitly indicated. As before, denote the white
contour sequence of a mobile $(\tau_n,\ell_n)$ by $(c_{i})_{0 \leq i
\leq n}$ where by definition $c_n = c_0$. Let $\xi_1$, $\xi_2,\ldots$
be a sequence of independent random variables identically distributed as
%
\begin{equation}
\P(\xi_1 = i) = 2^{-i-2},\qquad i=-1,0,1,\ldots.
\end{equation}
(This is a shifted geometric distribution with mean 0.)
The $\xi_i$ will have the role of jumps of the random walk. For each
black vertex $v \in\tau_n$, define the set $B_v \subseteq\mathbb
{N}$ by
%
\begin{equation}
B_v = \{i \in\mathbb{N} | c_{i-1} \sim v \mbox{ and }
c_i \sim v\},
\end{equation}
where $v \sim c_i$ means that $v$ and $c_i$ are nearest neighbors in
$\tau_n$. Define $S_m = \sum_{i=1}^m \xi_i$ and for any finite set $B
\subset\mathbb{N}$ let $S_B = \sum_{i\in B} \xi_i$. Define the conditioned
sequence of random variables
%
\begin{equation}
S_m^{\tau_n} = \bigl(S_m \mid S_{B_v} =
0 \mbox{ for all } v\in V^{\bullet}(\tau_n)\bigr),\qquad m=0,\ldots,n.
\end{equation}
A simple calculation similar to the one in (\ref{incr}) shows that
%
\begin{equation}
\label{eqSl} \bigl(S_m^{\tau_n} \bigr)_{m=0}^n
\stackrel{\mathrm{d}} {=} \bigl(\ell_n(c_m)
\bigr)_{m=0}^n.
\end{equation}
%
We have the following.

\begin{lemma}\label{Lineqshat}
Let $\tau_n$ be a fixed tree and let
$\hat{S}^{\tau_n}_n(t)$ be the continuous function on $[0,1]$ defined by
$\hat{S}^{\tau_n}_n(t) = n^{-1/2} S^{\tau_n}_{nt}$ when $t\in[0,1]$
and $nt$
is an integer, and extended by linear interpolation to all $t\in
[0,1]$. For
every $p\geq2$, there exists a constant $C_1(p)$ independent of
$n$ and $\tau_n$
such that
%
\begin{equation}
\label{ineqshat} \mathbb{E}\bigl|\hat{S}^{\tau_n}_n(t)-
\hat{S}^{\tau_n}_n(s)\bigr|^p \leq C_1(p)
|s-t|^{p/2}
\end{equation}
for any $0\leq s \leq t \leq1$.
\end{lemma}

\begin{pf}
First, consider the case when $s = k/n$ and $t = l/n$ for integers $k$
and~$l$. Suppose that $k < l$ and define $A = \{k+1,\ldots,l\}$ and
$A_v = A \cap B_v$, for every $v \in V^\bullet:= V^\bullet(\tau_n)$.
Then $A$ is the disjoint union of the $A_v$, $v \in V^\bullet$, and thus
%
\begin{equation}
S_l-S_k = S_A = \sum
_{v \in V^\bullet} S_{A_v}.
\end{equation}
Conditioning on $S_{B_v} = 0$ for all $v\in V^\bullet$ now yields
%
\begin{equation}
S^{\tau_n}_l-S^{\tau_n}_k = \sum
_{v \in V^\bullet} (S_{A_v} | S_{B_v} = 0).
\end{equation}
Define $Y_v = (S_{A_v} | S_{B_v} = 0)$ for every $v \in V^\bullet$,
and note that the random variables $Y_v$ are independent.
By \cite{legall:2011}, Lemma~1, there exists a constant $C_2(p)>0$
such that
for every $v$
%
\begin{equation}
\mathbb{E}|Y_v|^p \leq C_2(p)
|A_v|^{p/2}.
\end{equation}
Thus, by Rosenthal's inequality (see, e.g., \cite{gut:2005}, Theorem~3.9.1),
%
\begin{eqnarray}
&&\mathbb{E} \bigl|S^{\tau_n}_l-S^{\tau_n}_k\bigr|^p\nonumber\\
&&\qquad
= \mathbb{E} \biggl|\sum_{v
\in V^\bullet} Y_v
\biggr|^p \leq C_3(p) \sum_{v \in V^\bullet}
\mathbb {E}|Y_v |^p + C_4(p) \biggl(\sum
_{v \in V^\bullet} \mathbb{E}|Y_v |^2
\biggr)^{p/2}
\nonumber
\\[-8pt]
\\[-8pt]
\nonumber
&&\qquad\leq
\nonumber
C_5(p) \sum_{v \in V^\bullet}
|A_v|^{p/2} + C_6(p) \biggl(\sum
_{v \in V^\bullet} |A_v | \biggr)^{p/2}
\\
&&\qquad\leq C_7(p) \biggl(\sum_{v \in V^\bullet}
|A_v | \biggr)^{p/2} = C_7(p)
(l-k)^{p/2},\nonumber
\end{eqnarray}
which is equivalent to \eqref{ineqshat} in this case.
The case when $k/n \leq s \leq(k+1)/n$ follows directly since $\hat
{S}^{\tau_n}_n(t)$ is linear on $[k/n,(k+1)/n]$ and the general case
follows by splitting the interval $[s,t]$ into (at most) threes pieces
and using Minkowski's inequality.
\end{pf}

\begin{pf*}{Proof of Lemma~\ref{l:labeldist}}
We will prove an equivalent statement for $S^{\tau_n}_m$. For any
$t\in
[0,1)$ define the dyadic approximations $t_j = 2^{-j}\lfloor2^j
t\rfloor$,
$j=0,1,\ldots.$ Then $t_0=0$ and $t_j \rightarrow t$ as $j\rightarrow
\infty$. Since $\hat{S}_n^{\tau_n}$ is continuous, it holds that
$\hat{S}^{\tau_n}_n(t) = \sum_{j=0}^\infty
(\hat{S}^{\tau_n}_n(t_{j+1})-\hat{S}^{\tau_n}_n(t_j))$. Fix $p>2$. For
any $\varepsilon> 0$, by H\"{o}lder's inequality, letting $p'$ be the conjugate
exponent
%
\begin{eqnarray}
\bigl|\hat{S}^{\tau_n}_n(t)\bigr|^p &\leq& \Biggl(\sum
_{j=0}^\infty2^{-p'\varepsilon j}
\Biggr)^{p/p'} \sum_{j=0}^\infty2^{p\varepsilon j}
\bigl|\hat{S}^{\tau_n}_n(t_{j+1})-\hat{S}^{\tau_n}_n(t_j)\bigr|^p
\nonumber
\\[-8pt]
\\[-8pt]
\nonumber
&\leq& C_8(p,\varepsilon) \sum_{j=0}^\infty2^{p\varepsilon j}
\sum_{k=1}^{2^{j+1}}\bigl|\hat{S}^{\tau_n}_n
\bigl(k/2^{j+1}\bigr)-\hat{S}^{\tau
_n}_n
\bigl((k-1)/2^{j+1}\bigr)\bigr|^{p}.
\end{eqnarray}
The right-hand side is independent of $t$ so taking the supremum over
$t$ and then taking the expectation and using (\ref{ineqshat}) gives
%
\begin{eqnarray}
\mathbb{E} \sup_{t\in[0,1]} \bigl|\hat{S}^{\tau_n}_n(t)\bigr|^p
&\leq& C_8(p,\varepsilon) \sum_{j=0}^\infty2^{p\varepsilon j}
2^{j+1} C_{1}(p) 2^{-jp/2}
\nonumber
\\[-8pt]
\\[-8pt]
\nonumber
& =& C_{9}(p,
\varepsilon) \sum_{j=0}^\infty2^{(p\varepsilon+1 -p/2)j}.
\end{eqnarray}
By choosing $\varepsilon< (p/2-1)/p$, the estimate (\ref{labeldist}) follows
due to (\ref{eqSl}).
\end{pf*}

\begin{remark}
By \cite{billingsleyI}, Theorem~12.3 and (12.51),
Lemma~\ref{Lineqshat} implies also that the family of all random functions
$\hat{S}^{\tau_n}_n(t)$, where $n\in\mathbb{N}$
and $\tau_n$ ranges over all rooted
planar trees with $n$ edges, is tight in $C([0,1])$;
equivalently, we may consider $n^{-1/2}\ell_n(c_{nt})$, extended to
$t\in[0,1]$ by linear interpolation.
However, this family does not have a
unique limit in distribution as $n\to\infty$. For example, if $\tau_n$
is a
star, then
$\hat{S}^{\tau_n}_n(t)$ converges to $\sqrt2\mathbf{b}(t)$, where
$\mathbf
b$ is a Brownian bridge, while if $\tau_n$ is a path, with the root at one
endpoint, $\hat{S}^{\tau_n}_n(t)$ converges to
$(2/3)^{1/2}\mathbf{B}(t\wedge(1-t))$ where $\mathbf B$ is a standard
Brownian motion. And in many cases,
$\hat{S}^{\tau_n}_n(t)$ converges to 0;
if, for example, $\tau_n$ is a random binary tree, then
$n^{-1/4} S^{\tau_n}_{nt}$ converges in distribution. See, for
example, \cite{janson-marckert:2005}, and thus
$\hat{S}^{\tau_n}_n(t)$ is typically of the order $n^{-1/4}$.
\end{remark}

\section{Another useful bijection and simply generated trees} \label
{s:anotherb}
The coloring of the vertices in the mobiles is simply a bookkeeping
device which groups together vertices in every second generation. We
will continue referring to black and white vertices in trees even when
no labels are assigned to white vertices. There exists a useful
bijection from the set of trees $\Gamma_n$ to itself which maps white
vertices to vertices of degree 1 and black vertices of degree $k\geq1$
to vertices of degree $k+1$. We will denote the bijection by $\mathcal
{G}_n$. The bijection can be described informally in the following way:
Start with a tree with vertices colored black and white as described
above, the root being white. It will be mapped to a new tree which has
the same vertex set as the old one but different edges. First consider
the root, $r$, say of degree $i$ and denote its black children by
$r_1,\ldots,r_{i-1}$. Begin by attaching a half-edge to $r_{1}$ which
becomes the root of the new tree. Then connect $r_{j}$ to $r_{j+1}$
with an edge for $1 \leq j\leq i-1$ and finally
connect $r_{i-1}$ to
the root $r$. Continue in the same way recursively for each of the
subtrees attached to each of the $r_{j}$. More precisely, for a given
white vertex $u\neq r$ of degree $k$ denote its parent by $u_{0}$ and
its children by $u_{1},\ldots,u_{k-1}$. Insert an edge between $u_j$
and $u_{j+1}$ for $0 \leq j < k-1$ if possible (i.e., if $k>0$), and
finally connect $u_{k-1}$ to $u$; see Figure~\ref{f:treetotree}.
%
\begin{figure}

\includegraphics{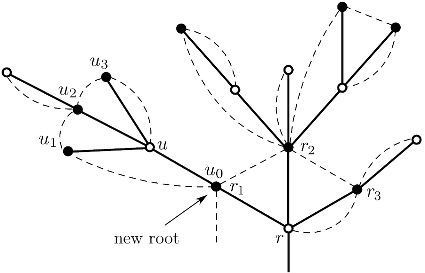}

\caption{A diagram describing the bijection from $\Gamma_n$ to itself
which sends white vertices to vertices of degree 1 and black vertices
of degree $k$ to vertices of degree $k+1$.} \label{f:treetotree}
\end{figure}

To see that $\mathcal{G}_n$ is a bijection, we describe here its
inverse. Start with a tree with all vertices black except the leaves
which are white.
Let $(a_i)_{i \geq0}$ be the contour sequence of the tree. If $a_j$ is
a leaf let $\eta(j)$ denote the maximum number such that $a_j,
a_{j+1},\ldots,a_{j+\eta(j)}$ all lie on the path from $a_j$ to the
root. Now, for each white $a_j$ insert an edge between $a_j$ and
$a_{j+k}$ for $1 \leq k \leq\eta(j)$ and remove the edges of the
original tree. Let the last white vertex (within one period $[0,2n)\cap
\mathbb{Z}$) in the contour sequence be the root of the resulting tree.
In the process, the degree of each black vertex is reduced by one and
the degree of a white vertex $a_j$ becomes $\eta(j)$ with the exception
of the root in which case the degree becomes $\eta+1$.

The usefulness of the bijection $\mathcal{G}_n$ is that it gives a
simple description of the probability distribution $\tilde{\nu}_n$. Let
$\nu_n$ be the push-forward of $\tilde{\nu}_n$ by $\mathcal{G}_n$. By
(\ref{eq:bolmob}) and the properties of $\mathcal{G}_n$,
%
\begin{equation}
\label{nu} \nu_n(\tau_n) = 2 Z_n^{-1}
\prod_{v\in V(\tau_n)} w_{\deg(v)-1},
\end{equation}
where we recall that $w_i$ was defined in (\ref{eq:defp}). The
convenient thing is that now all vertices are treated equally. The
probability measure $\nu_n$ describes simply generated trees,
originally introduced by Meir and Moon \cite{meir:1978} and has since
been studied extensively; see, for example, \cite{janson:2012} and
references therein.

For the weights (\ref{A1}) in case \eqref{cond:A1} in the
\hyperref[s:intro]{Introduction}, we
define the
probabilities
%
\begin{equation}
\label{probo} p_i = \frac{w_i}{g(1)};
\end{equation}
thus, for $i\ge1$, with $\LL(i)= g(1)^{-1}L(i)$,
%
\begin{equation}
\label{probox} p_i=\LL(i) i^{-\gb}.
\end{equation}
We let $\mathbb{P}_{p}$ be the law of a Galton--Watson tree with offspring
distribution $(p_i)_{i\geq0}$; see, for example, \cite{athreya,janson:2012}.
Note that the expected number of offspring of an individual in the
Galton--Watson process is equal to $g'(1)/g(1) = \mean$. We will furthermore
denote the variance of the number of offspring by
%
\begin{equation}
\sigma^2 = g''(1)/g(1)+\mean(1-\mean),
\end{equation}
which may be finite or infinite depending on the value of $\beta$. The
measure $\nu_n$ viewed as a measure on the set of finite trees is in
this case equal to the measure $\mathbb{P}_{p}(\cdot| |\tau|=n)$,
where $\tau$ denotes a finite tree. In case (\ref{cond:A2}), $\nu_n$
has no such equivalent description in terms of a Galton--Watson process.

Using the bijection $\mathcal{G}_n$, one can translate known results on
simply generated trees to the trees distributed by $\tilde{\nu}_n$. We
will now introduce some notation and state a few technical results
needed later on, some of which are interesting by themselves. In a
random tree $\tau_n$ distributed by $\tilde{\nu}_n$ select a black
vertex of maximum degree in some prescribed way (e.g., as the first
such vertex encountered in the lexicographical order) and denote it by
$s$. Denote the degree of $s$ by $\Delta_n$ and the white vertices
surrounding $s$ by $s_{0},s_{1},\ldots,s_{\Delta_n-1}$ in a clockwise
order, taking $s_0$ as the parent of $s$. For more compact notation, we
do not explicitly write the dependency of $s$ and $s_i$ on $n$.

Denote by $\tau_{n,0}$ the tree which consists of all vertices in
$\tau
_n$ apart from $s$ and its descendants. Let $\tau_{n,i}$ be the tree
consisting of $s_i$ and its descendants, $1 \leq i \leq\Delta_n-1$.
Furthermore, define $N_{n,i}^\circ$ as the number of white vertices in
$\tau_{n,i}$. Write $\tau_n' = \mathcal{G}_n(\tau_n)$ and let $s'$ be
the vertex in $\tau_n'$ corresponding to the vertex $s$ in $\tau_n$.
Then $\deg(s') = \Delta_n+1$. Define the subtrees $\tau_{n,i}'$ around
$s'$ in $\tau_n'$ in an analogous way as above where $0 \leq i \leq
\Delta_n$. It is then simple to check that
%
\begin{equation}
\label{masssum} |\tau_{n,0}| = \bigl|\tau_{n,0}'\bigr| +
\bigl|\tau_{n,\Delta_n}'\bigr|+1 \quad\mbox{and} \quad|\tau_{n,i}| =\bigl |
\tau_{n,i}'\bigr|
\end{equation}
for $1 \leq i \leq\Delta_n -1$. This is the key relation used to
translate results from the simply generated trees to the mobiles.

Let $Y=(Y_t)_{t\geq0}$ be the spectrally positive stable process with Laplace
transform $\mathbb{E}(\exp(-\lambda Y_t))
=\exp (t\lambda^{2\wedge(\beta-1)} )$.
(This is a L\'{e}vy process with no negative jumps; the L\'{e}vy
measure is
$\gG(-\ga)^{-1} x^{-\ga-1}\,d x$ on $x>0$, where $\ga=2\wedge(\gb
-1)\in(1,2]$.
See, for example, \cite{Bertoin} and \cite{Zolotarev}.)
Denote by $\mathbb{D}([0,1])$ the set of
c\`{a}dl\`{a}g functions $[0,1]\rightarrow\mathbb{R}$ with the Skorohod
topology; see \cite{billingsleyI}, Section~14. We have the following
proposition for the case \eqref{A1}, where $0<\mean<1$.

\begin{proposition} \label{p:scgw}
For the weights \eqref{A1}, the tree distributed by $\tilde{\nu}_n$
has the
properties that
\begin{longlist}[(1)]
\item[(1)]
%
\begin{equation}
\frac{\Delta_n}{n} \mathop{\xrightarrow}_{n\rightarrow\infty
}^{\inp} 1-\mean.
\end{equation}
\item[(2)]
%
\begin{equation}
\frac{N_n^\circ}{n} \mathop{\xrightarrow}_{n\rightarrow\infty
}^{\inp}
p_0
\end{equation}
with $p_0=1/g(1)$ defined in \eqref{probo}.

\item[(3)]
For any fixed $i\geq0$, $|\tau_{n,i}|$
converges in
distribution as $n\rightarrow\infty$ to a finite random variable.
For $i\ge1$, the limit equals $|\tau|$, where $\tau$ is a Galton--Watson
tree with offspring distribution $(p_i)_{i\ge0}$.

\item[(4)]
There exists a slowly varying function $L_1(n)$ such that for
$C_n = L_1(n)\times  n^{{1}/{(2\wedge(\beta-1))}}$ the following weak convergence
holds in $\mathbb{D}([0,1])$:
%
\begin{equation}
\label{endist} \biggl(\frac{\sum_{i=1}^{\lfloor(\Delta_n-1) t\rfloor
}N_{n,i}^\circ
-({p_0}/{1-\mean})\Delta_n t }{C_{n }} \biggr)_{0\leq t \leq1} \mathop{
\xrightarrow}_{n\rightarrow\infty}^{\ind} (Y_t )_{0\leq t
\leq1}.
\end{equation}
\item[(5)]
It holds that
%
\begin{equation}
\frac{1}{C_n} \sup_{1\leq i \leq\Delta_n-1} N_{n,i}^\circ
\mathop{\xrightarrow}_{n\rightarrow\infty}^{\ind} V
\end{equation}
with $C_n$ from part (4) and the random variable $V=\max_{0\le
t\le1}\gD Y_t$.
\end{longlist}
\end{proposition}

\begin{pf}
Part (1) follows from the corresponding result for simply
generated trees which was originally proven in \cite{jonsson:2011} in the
case of an asymptotically constant slowly varying function $L$ in
(\ref{A1}) and then
in \cite{kortchemski:2012} for a general slowly varying function $L$.

Part (2) follows from \cite{janson:2012}, Theorem~7.11(ii), since
the number of white vertices $N_n^\circ$ in the tree $\tau_n$
equals the number of leaves in the simply generated tree $\tau'_n$,
via the
bijection
$\mathcal{G}_n$.

For part (3), we note that the simply generated trees
distributed by $\nu_n$ converge locally toward an infinite random tree;
see \cite{jonsson:2011}, Theorem~5.3, in the case of an asymptotically
constant slowly varying function $L$ and \cite{janson:2012}, Theorem~7.1,
for the most general case. Local convergence of the trees distributed by
$\tilde{\nu}_n$ follows and the result in part (3)
is then an immediate consequence; see the arguments in the proof of Theorem~3(iii) in
\cite{kortchemski:2012}.

Part (4) requires some explanation. We will prove a
corresponding statement for the simply generated trees distributed by
$\nu_n$. Recall the notation $\tau_n$ for (colored) trees distributed by
$\tilde{\nu}_n$ and
$\tau'_n$ for (conditioned Galton--Watson) trees distributed by $\nu
_n$ as
explained in the paragraph above (\ref{masssum}).
First of all, note
that the number of white vertices in $\tau_{n,i}$, which is denoted by
$N_{n,i}^\circ$, corresponds to the number of leaves in the trees
$\tau_{n,i}'$ for $1\leq i \leq\Delta_n -1$.

Recall that $\mathbb{P}_{p}$ is the law of a Galton--Watson process with
the offspring distribution $(p_i)_{i\geq0}$ defined in
(\ref{probo}). Denote by $N$ the total progeny (number of vertices) of the
Galton--Watson process distributed by $\mathbb{P}_{p}$ and denote the
random number of leaves by $N^{(0)}$.
It is well known that $\mathbb{E}(N) = 1/(1-\mean)$ (see,
e.g., \cite{athreya}), and
furthermore,
%
\begin{equation}
\label{en0} \E N^{(0)}=\frac{p_0}{1-\mean}=p_0\E N;
\end{equation}
in fact, the expected number of vertices in generation $m\ge0$ is
$\mean^m$,
and the expected number of leaves among them is $p_0\mean^m$, where summing
over all $m\ge0$ yields \eqref{en0}.
This explains the linear term in \eqref{endist}.

Kortchemski \cite{kortchemski:2012}, Theorem~4, proved a
convergence result in $\mathbb{D}([0,1])$
which in our notation can be written as
%
\begin{equation}
\label{kortweak} \biggl(\frac{\sum_{i=1}^{\lfloor(\Delta_n-1) t\rfloor}(|\tau
_{n,i}'|+1)-({1}/{(1-\mean)})\Delta_n t }{B_{n}'} \biggr)_{0\leq t
\leq
1} \mathop{
\xrightarrow}_{n\rightarrow\infty}^{\ind} ({Y}_t
)_{0\leq t
\leq1},
\end{equation}
where $B_n' = L_2(n)n^{{1}/{(2\wedge(\beta-1))}}$ for some slowly varying
function $L_2$. 
The main idea
of Kortchemski's proof is to use the fact that for $n$ large, the subtrees
$\tau_{n,i}'$ become asymptotically
independent copies of a Galton--Watson process with law
$\mathbb{P}_{p}$, and thus $|\tau_{n,i}'|+1$ appearing in the sum in
(\ref{kortweak}) can be replaced by a sequence $(N_i)_{i\geq1}$
of independent random variables distributed as $N$.
(This is shown in \cite{kortchemski:2012} as a consequence of a
corresponding result for random walks by
Armend\'{a}riz and Loulakis \cite{ArmendarisLoulakis}.)
Furthermore, it is well known
(see, e.g., \cite{Otter,Kemperman1,Kemperman2}, \cite{Pitman}, Section~6.1, \cite{janson:2012}, Theorem~15.5) that
if $\xxi_i$, $i=1,2,\ldots,$ is a sequence of independent random
variables with the distribution $(p_i)_{i\ge0}$,
and we let $\xS_n=\sum_{i=1}^n\xxi_i$, then
%
\begin{equation}
\label{jk} \P(N=n)=\frac{1}n \P(\xS_n=n-1).
\end{equation}
Moreover, from the tail behavior \eqref{probox} of $p_i=\P(\xi=i)$,
it follows that,
recalling that $\E\xxi_i=\mean$,
%
\begin{eqnarray}
\label{sw} \P(\xS_n=n-1) &=& \P \bigl(\xS_n-n
\mean=n(1-\mean)-1 \bigr)
\nonumber
\\[-8pt]
\\[-8pt]
\nonumber
& &=n\bigl(1+o(1)\bigr)\P\bigl(\xxi_1=\bigl\lfloor n(1-\mean)-1\bigr
\rfloor\bigr)
\end{eqnarray}
as ${n\to\infty}$, see \cite{DenisovDS} for more general statements.
(In our case, \eqref{sw} follows also directly by a modification of the
proof of \cite{janson:2012}, Theorem~19.34.)
Combining \eqref{jk}, \eqref{sw} and \eqref{probox}, we obtain
%
\begin{equation}
\label{pi}\qquad \P(N=n) 
= \bigl(1+o(1)\bigr) (1-\mean)^{-\gb}
\LL(n) n^{-\gb} = \bigl(1+o(1)\bigr) (1-\mean)^{-\gb}
p_n,
\end{equation}
so the distribution of $N$ also obeys \eqref{A1} (with a different $L$),
which by standard results (see, e.g., \cite{FellerII}, Section XVII.5)
implies that $N$ is in the domain of attraction of
a spectrally positive stable distribution of index $\ga=2\wedge(\gb
-1)$, and
thus
%
\begin{equation}
\label{jcd} \biggl(\frac{\sum_{i=1}^{\lfloor n t \rfloor}N_i-({1}/{(1-\mean)})nt
}{B_n'} \biggr)_{0\leq t \leq1} \mathop{
\xrightarrow}_{n\rightarrow
\infty
}^{\ind} ({Y}_t )_{0\leq t \leq1}
\end{equation}
for a suitable
$B_n'=L_2(n)n^{{1}/{(2\wedge(\beta-1))}}$.
We refer to \cite{kortchemski:2012} for further details, and for the
arguments using \eqref{jcd} to show \eqref{kortweak}.

Going through Kortchemski's proof, one sees that the latter arguments
apply in
our case also
if we replace $\mathscr{Z}^{(k)}$ in \cite{kortchemski:2012} by
$ (C_k^{-1} (\sum_{i=1}^{\lfloor kt\rfloor} N_i^{(0)} - \frac
{p_0}{1-\mean} kt ) )_{0\leq t\leq\eta}$
and the problem is reduced to showing that if
$(N_i,N_i^{(0)})_{i\geq1 }$ is a sequence of
independent random vectors distributed as $(N,N^{(0)})$, then
%
\begin{equation}
\label{N0conv} \biggl(\frac{\sum_{i=1}^{\lfloor n t \rfloor}N_{i}^{(0)}-
({p_0}/{(1-\mean)})n t} {
C_n} \biggr)_{0\leq t \leq1} \mathop{
\xrightarrow}_{n\rightarrow\infty}^{\ind} (\hat{Y}_t
)_{0\leq t \leq1},
\end{equation}
where $\hat{Y}$ has the same distribution as $Y$, and that this
holds jointly with (\ref{jcd}).
(Joint convergence is used in the analogue of
\cite{kortchemski:2012}, (31), in the proof; however, the joint distribution
of $(Y,\hat Y)$ does not influence the result \eqref{endist}.)
The proof of part~(4) is thus completed by Lemma~\ref
{Ljoint} below.

Finally, part (5) follows from part (4);
see the proof of Corollary~2 in \cite{kortchemski:2012}.
\end{pf}

\begin{remark}
Actually, it would suffice to prove \eqref{N0conv} separately; this and~\eqref{jcd} show in particular that the left-hand sides are tight in
$\mathbb{D}([0,1])$, which implies that they are jointly tight in
$\mathbb{D}([0,1])\times\mathbb{D}([0,1])$,
and we can obtain the desired joint convergence by considering suitable
subsequences; this is enough to show \eqref{endist} for the full sequence
since the result does not depend on the joint distribution of
$ ({Y}_t )_{0\leq t \leq1}$ and $ (\hat{Y}_t
)_{0\leq t \leq1}$.
We can show \eqref{N0conv} by the same
standard results as for \eqref{jcd} together with the estimate
%
\begin{equation}
\label{n0tail} \P\bigl(N^{(0)}=n\bigr)\sim c \LL(n) n^{-\gb}
\end{equation}
for some $c>0$,
see Lemma~\ref{LN0},
which shows that
the distribution of $N^{(0)}$ has the same tail behavior as $N$ and
$(p_i)_{i\ge0}$.
This thus yields an alternative proof of
Proposition~\ref{p:scgw}(4).
\end{remark}

Before stating and proving Lemma~\ref{Ljoint} used above, we give
another lemma.

\begin{lemma}\label{L1}
For the weights \eqref{A1},
with notation as above, as ${n\to\infty}$,
%
\begin{equation}
\label{l1b} \P \bigl( \bigl|N^{(0)}-p_0 N \bigr|\ge n \bigr) = o
\bigl(\LL (n)n^{1-\gb} \bigr) =o(n p_n) =o \bigl(\P(N\ge n)
\bigr).
\end{equation}
\end{lemma}

\begin{pf}
Note first that \eqref{pi} and \eqref{probox} imply,
by a standard calculation \cite{BinghamGoldieTeugels:1987},
%
\begin{eqnarray}
\label{pix} \P(N\ge n) &=& \bigl(1+o(1)\bigr) (1-\mean)^{-\gb} (
\gb-1)^{-1}\LL(n) n^{1-\gb}
\nonumber
\\[-8pt]
\\[-8pt]
\nonumber
&=& \bigl(1+o(1)\bigr) (\gb-1)^{-1}(1-\mean)^{-\gb} n
p_n.
\end{eqnarray}

Let $a>0$. Since $|N^{(0)}-p_0N|\le N$,
%
\begin{eqnarray}
\label{ems} &&\P \bigl(\bigl|N^{(0)}-p_0N\bigr|\ge n \bigr)
\nonumber
\\[-8pt]
\\[-8pt]
\nonumber
&&\qquad \le \P(N
\ge an)
+ \P\biggl(\bigl|N^{(0)}-p_0N\bigr|\ge\frac{1}a N
\mbox{ and } N\ge n\biggr).
\end{eqnarray}
Let $\eps>0$.
By \cite{janson:2012}, Theorem~7.11,
$(N^{(0)}\mid N=n)/n \pto p_0$ as ${n\to\infty}$. Thus,
$\P (|N^{(0)}-p_0N|\ge a^{-1}N \mid N=n )<\eps$
if $n$ is large enough, and for such $n$,
%
\begin{eqnarray}
\label{emsx}&& \P\biggl(\bigl|N^{(0)}-p_0N\bigr|\ge\frac{1}a
N \mbox{ and } N\ge n\biggr)
\nonumber\\
&&\qquad =\sum_{m=n}^\infty \P\biggl(\bigl|N^{(0)}-p_0N\bigr|
\ge\frac{1}a N\Bigm| N=m\biggr)\P(N=m)
\\
&&\qquad \le\eps\P(N\ge n).\nonumber
\end{eqnarray}
Thus, \eqref{ems} yields, for large $n$,
%
\begin{equation}
\P \bigl(\bigl|N^{(0)}-p_0N\bigr|\ge n \bigr) \le\P(N\ge
an) + \eps\P (N\ge n),
\end{equation}
which
by \eqref{pix} yields, with $C=(\gb-1)^{-1}(1-\mean)^{-\gb}$,
%
\begin{equation}
\P \bigl(\bigl|N^{(0)}-p_0N\bigr|\ge n \bigr) \le\bigl(1+o(1)\bigr)C
\bigl(a^{1-\gb}+\eps \bigr)\LL(n)n^{1-\gb}.
\end{equation}
Since we may choose $a$ arbitrarily large and $\eps$ arbitrarily small,
\eqref{l1b} follows.
\end{pf}

\begin{lemma}
\label{Ljoint}
The limits
\eqref{jcd} and \eqref{N0conv},
in distribution in $\mathbb{D}([0,1])$, hold jointly.
\end{lemma}

\begin{pf}
Suppose first that
the offspring distribution \eqref{probox} has finite variance.
[This implies $\gb\ge3$ by \eqref{pi}.]
It then follows from
\eqref{pi} that $N$ and $N^{(0)}\le N$ have finite variances and
by a two-dimensional version of Donsker's theorem,
the result
follows with $2^{-1/2}Y_t$ and $2^{-1/2}\hat Y_t$ two different (dependent)
standard Brownian motions, and
$B'_n = \sqrt{\textrm{Var}(N) n /2}$,
$C_n = \sqrt{\textrm{Var}(N^{(0)}) n /2}$.

Suppose now instead that the variance of the offspring distribution is
infinite; then $\E N^2=\infty$. We follow \cite{FellerII}, Section XVII.5,
and let $\mu(x)$ be the truncated moment function
%
\begin{equation}
\mu(x) = \E\bigl(N^2 \mathbf1\{N\le x\}\bigr).
\end{equation}
Then $\mu(x)\to\infty$ as $x\to\infty$. Moreover,
by \cite{FellerII}, Theorem XVII.5.2 and\break XVII.(5.23),
$\mu(x)$ is regularly
varying with exponent $2-\ga=(3-\gb)\vee0$, and
\eqref{jcd} holds with
$n\mu(B_n')/(B_n')^2\to C$
for some constant $C$.

If we similarly define the truncated moment function,
%
\begin{equation}
\mu_1(x) = \E \bigl(\bigl(N^{(0)}-p_0 N
\bigr)^2 \mathbf1\bigl\{\bigl|N^{(0)}-p_0N\bigr|\le x\bigr
\} \bigr),
\end{equation}
it follows easily by \eqref{l1b} [and $\mu(x)\to\infty$] that, as
$x\to\infty$,
%
\begin{equation}
\mu_1(x) = o\bigl(\mu(x)\bigr)
\end{equation}
and thus
%
\begin{equation}
\frac{n\mu_1(B_n')}{(B_n')^2} =o\biggl(\frac{n\mu(B_n')}{(B_n')^2}\biggr)\to0,\qquad {n\to\infty}.
\end{equation}
It follows by minor modifications of the arguments in
\cite{FellerII}, Section XVII.5,
that
%
\begin{equation}
\frac{\sum_{i=1}^{n}(N_{i}^{(0)}-p_0N_i)}{B_n'}\pto0.
\end{equation}
Moreover, by \cite{Kallenberg}, Theorem~16.14,
or by symmetrization and a stopping time argument, it follows that
%
\begin{equation}
\sup_{0\le t\le1}\Biggl |\sum_{i=1}^{\lfloor n t \rfloor }
\bigl(N_{i}^{(0)}-p_0N_i
\bigr)\biggr|\Big/B_n'\pto0,
\end{equation}
and thus \eqref{jcd} implies that
\eqref{N0conv} holds jointly with
$\hat Y_t=Y_t$ and $C_n=p_0B_n'$.
(Note that $\hat Y=Y$ when the offspring variance is infinite, but not when
it is finite.)
\end{pf}

For the case \eqref{A2}, where $\mean=0$,
the proposition below follows immediately from
\cite{janson:2011}, Theorems 2.4--2.5 and Remark~2.9.

\begin{proposition} \label{p:super}
For $w_i = (i!)^\alpha$, $\alpha> 0$, the tree distributed by
$\tilde{\nu}_n$ has the following properties:
\begin{longlist}[(1)]
\item[(1)]
For $\alpha> 1$,
%
\begin{equation}
n-\Delta_n \mathop{\xrightarrow}_{n\rightarrow\infty}^{\inp} 0.
\end{equation}
For $\alpha= 1$,
%
\begin{equation}
n - \Delta_n \mathop{\xrightarrow}_{n\rightarrow\infty}^{\ind}
\operatorname{Pois}(1).
\end{equation}
For $\alpha< 1$
%
\begin{equation}
n - \Delta_n = O\bigl(n^{1-\alpha}\bigr)
\end{equation}
with probability tending to $1$ as $n\rightarrow\infty$.
\item[(2)]
%
\begin{equation}
\frac{N_n^\circ}{n} \mathop{\xrightarrow}_{n\rightarrow\infty
}^{\inp} 1.
\end{equation}
\item[(3)]
The vertex $s$ is the unique black child of the root $r$ and
%
\begin{equation}
\sup_{1\leq i \leq\Delta_n-1} N_{n,i}^\circ\leq\lfloor1/
\alpha \rfloor \vee1
\end{equation}
with probability tending to $1$ as $n\rightarrow\infty$.
\end{longlist}
\end{proposition}

The propositions above along with the correspondence between degrees of
faces in the planar maps and degrees of black vertices in the mobiles
show that a unique face of degree roughly equal to $(1-\mean)n$ appears
in the planar maps $M_n$ with probability tending to 1 as $n\rightarrow
\infty$.

\section{Label process on mobiles} \label{s:label}

Let $\theta_n$ be a random mobile distributed by $\tilde{\mu}_n$,
and denote
by $N^\circ_n$ the random number of white vertices in $\theta_n$.
Order the
white vertices in a lexicographical order $v_{0},v_{1},\ldots
,v_{N^\circ_n}$
(taking $v_{N^\circ_n} = v_{0}$). Again we do not write explicitly the
dependency of $v$ and $v_i$ on $n$. Define the label process $L_n\dvtx \{0,1,\ldots,N^\circ_n\} \rightarrow\mathbb{Z}$ by $L_n(i) =
\ell_n(v_{i})$. Extend $L_n$ to a function on $[0,N^\circ_n]$ by linear
interpolation.

Denote the set of continuous functions from $[a,b]$ to $\mathbb{R}$ by
$C([a,b])$ equipped with the topology of uniform convergence.
Let $\mathbf{b}$ be the standard Brownian bridge on $[0,1]$, starting and
ending at 0. We will in this section prove the following result.

\begin{theorem} \label{th:inv}
For the weights \eqref{A1} and \eqref{A2}, it holds that
%
\begin{equation}
\label{inv} \biggl(\frac{1}{\sqrt{2(1-\mean)n}}L_n\bigl(t
N^\circ_n\bigr) \biggr)_{0 \leq t
\leq
1} \mathop{
\xrightarrow}_{n\rightarrow\infty}^{\ind} \bigl(\mathbf {b}(t)
\bigr)_{0
\leq t \leq1}
\end{equation}
with convergence in distribution in
$C([0,1])$.
\end{theorem}

Since the label function encodes information on distances,
cf. (\ref{distances}), this result shows that the diameter of the maps grows
like $n^{1/2}$.
More precisely, we can translate Theorem~\ref{th:inv} to a result on
distances to
the marked vertex $\rho$.
Define the distance process $D_n\dvtx \{0,1,\ldots,N^\circ_n\} \rightarrow
\mathbb{Z}$ by $D_n(i)=d(v_i,\rho)$.
Extend $D_n$ to a function on $[0,N^\circ_n]$ by linear
interpolation, and then to a function on $\bbR$ with period $N^\circ_n$.
By \eqref{distances},
%
\begin{equation}\qquad
\label{dl} D_n(t)=L_n(t)-\ell_n(\rho)
=L_n(t)-\min_{0\le s\le N_n^\circ} L_n(s)+1,\qquad 0\le t\le
N_n^\circ.
\end{equation}
Further, let $v_{i_*}$ be the first white vertex (in our ordering) that
is a
neighbor of $\rho$, that is, $i_*=\min\{i\dvtx \ell_n(v_i)=\min_j \ell
_n(v_j)\}$.

\begin{theorem} \label{th:Dinv}
For the weights \eqref{A1} and \eqref{A2}, it holds that
%
\begin{equation}
\label{dinv} \biggl(\frac{1}{\sqrt{2(1-\mean)n}}D_n\bigl(t
N^\circ_n+i_*\bigr) \biggr)_{0
\leq t
\leq1} \mathop{
\xrightarrow}_{n\rightarrow\infty}^{\ind} \bigl(\mathbf {e}(t)
\bigr)_{0 \leq t \leq1}
\end{equation}
with convergence in distribution in
$C([0,1])$.
\end{theorem}

\begin{pf}
The minimum of $\mathbf{b}$ is a.s. attained at a unique point,
$\brmin$ say,
and $\brmin$ is
uniformly distributed on $[0,1]$;
moreover, by
Vervaat's
theorem \cite{vervaat:1979}, if this minimum is subtracted from
$\mathbf{b}$
and the bridge is shifted (periodically) such that the minimum is
located at
0 one obtains a standard Brownian excursion $\mathbf{e}$ on $[0,1]$;
see also \cite{biane:1986}.

By Skorohod's representation theorem, we may assume that the
convergence in~\eqref{inv} holds a.s. Since the minimum point $U$ is unique, it follows
that the minimum point $v_{i_*}/N^\circ_n$ of the left-hand side converges
to $U$ a.s.
(The minimum point $v_{i_*}$ is typically not unique.
We chose the first minimum
point, but any other choice would also converge to $U$ a.s.)
The desired convergence \eqref{dinv} now follows from~\eqref{dl}, \eqref{inv}
and Vervaat's theorem.
\end{pf}

We start by introducing some notation and proving a couple of lemmas
before proceeding to the proof of Theorem~\ref{th:inv}. Begin by
considering only the part of the label process which surrounds the
vertex $s$, a black vertex of maximum degree. Let $s_{0}$ be the white
parent of $s$ and let $s_{i}$ be its $i$th white child in clockwise
order from $s_{0}$, where $i=1,\ldots,\Delta_n$ with the convention
that $s_{\Delta_n} = s_{0}$. Define the function $L^\star_n\dvtx \{
0,1,\ldots,\Delta_n\} \rightarrow\mathbb{Z}$ by $L^\star_n(i) =
\ell
_n(s_{i})$. As before, extend $L^\star_n$ to a continuous function on
$[0,\Delta_n]$ by linear interpolation.

\begin{lemma} \label{l:inv1}
For the weights \eqref{A1} and \eqref{A2}, it holds that
%
\begin{equation}
\label{inv1} \biggl(\frac{1}{\sqrt{2(1-\mean)n}}L^\star_n(t
\Delta_n) \biggr)_{0
\leq t
\leq1} \mathop{\xrightarrow}_{n\rightarrow\infty}^{\ind}
\bigl(\mathbf {b}(t) \bigr)_{0 \leq t \leq1}
\end{equation}
with convergence in distribution in
$C([0,1])$.
\end{lemma}

\begin{pf}
Let $\theta_n = (\tau_n,\ell_n)$ be a mobile distributed by
$\tilde{\mu}_n$. By Propositions~\ref{p:scgw}(1) and
\ref{p:super}(1),
$\Delta_n/n \mathop{\xrightarrow}^{\inp} 1-\mean$ as $n\rightarrow
\infty
$. Using
Skorohod's representation theorem, we may construct $\Delta_n$ and
$L_n$ on a
common probability space such that this convergence holds almost surely,
that is,
%
\begin{equation}
\label{degas} \Delta_n/n \mathop{\xrightarrow}_{n\rightarrow\infty}^{\as}
1-\mean.
\end{equation}
In the following, we will assume that this holds.

The label process $L^\star_n$, evaluated on the integers, is a random walk
of length $\Delta_n$ having jump probabilities $\omega(k) = 2^{-k-2}$,
$k=-1,0,1,\ldots,$ starting at $L^\star_n(0)=\ell_n(s_0)$ and
conditioned to
end at $\ell_n(s_0)$; see \cite{legall:2011}, Section~3.3. It
follows from
Propositions \ref{p:scgw}(3) and \ref{p:super}(3)
that
$n^{-1/2}\ell_n(s_0) \mathop{\xrightarrow}^{\inp} 0$ as
$n\rightarrow\infty
$. The
jump distribution has mean 0 and
variance $\sum_{k=-1}^\infty k^2 \omega(k) = 2$. The
result now follows by a conditional version of Donsker's invariance theorem;
see, for example, \cite{bettinelli:2010}, Lemma~10, for a detailed proof.
\end{pf}

\begin{lemma} \label{l:shift}
Let $f_n,g_n\dvtx A_n \to[0,\Delta_n]$ be random functions, for some
(possibly random) set $A_n$. If
%
\begin{equation}
\label{condfg} \sup_{x\in A_n} n^{-1}\bigl|f_n(x)-g_n(x)\bigr|
\pto0
\end{equation}
then
%
\begin{equation}
\label{supfg1} n^{-1/2}\sup_{x\in A_n}\bigl|L_n^{\star}
\bigl(f_n(x)\bigr)-L_n^{\star}\bigl(g_n(x)
\bigr)\bigr| \mathop{\xrightarrow}_{n\rightarrow\infty}^{\inp} 0.
\end{equation}
\end{lemma}

\begin{pf}
By the triangle inequality,
\begin{eqnarray*}
&&\bigl(2(1-\mean)n\bigr)^{-1/2}\sup_{x\in A_n}
\bigl|L_n^{\star
}\bigl(f_n(x)\bigr)-L_n^{\star}
\bigl(g_n(x)\bigr)\bigr |
\\
&& \qquad\leq\sup_{x\in A_n} \bigl|\mathbf{b}\bigl(f_n(x)/
\gD_n\bigr)-\mathbf {b}\bigl(g_n(x)/\gD_n
\bigr) \bigr|
\\
&&\qquad\quad{} + \sup_{x\in A_n} \bigl|\bigl(2(1-\mean)n
\bigr)^{-1/2}L_n^{\star
}\bigl(f_n(x)\bigr)-
\mathbf{b}\bigl(f_n(x)/\gD_n\bigr) \bigr|
\\
&& \qquad\quad{}+ \sup_{x\in A_n}\bigl |\bigl(2(1-\mean)n
\bigr)^{-1/2}L_n^{\star}\bigl(g_n(x)\bigr)-
\mathbf{b}\bigl(g_n(x)/\gD _n\bigr) \bigr|.
\end{eqnarray*}
The first term converges to zero in probability by (\ref{condfg}) and the
fact that
$\mathbf{b}$ is continuous on $[0,1]$, and hence uniformly continuous. The
other terms converge to zero by Lemma~\ref{l:inv1}, assuming as we may (by
Skorohod's representation theorem)
that~\eqref{inv1} holds a.s.
\end{pf}

\begin{lemma}\label{L4}
As ${n\to\infty}$,
%
\begin{equation}
n^{-1/2} \sup_{0 \leq i \leq\Delta_n-1} \sup_{v \in\tau_{n,i}}\bigl|
\ell_{n}(v)-\ell_{n}(s_i)\bigr| 
\pto0.
\end{equation}
\end{lemma}

\begin{pf}
Write the left-hand side as $n^{-1/2}K$.
Choose $\delta>0$ with $1-\delta>1/(2\wedge(\gb-1))$, and choose
$p>2/\delta$.
We condition on $\tau_n$
and obtain, by using Lemma~\ref{l:labeldist} for each subtree
$\tau_{n,i}$ separately,
%
\begin{eqnarray}
\label{qwe} \E\bigl(K^p\mid\tau_n\bigr) &\le&\sum
_{i=0}^{\gD_n-1}\E\sup_{v \in\tau_{n,i}}\bigl|
\ell _{n}(v)-\ell_{n}(s_i)\bigr|^p \le
\sum_{i=0}^{\gD_n-1}C(p) \bigl(N^\circ_{n,i}
\bigr)^{p/2}
\nonumber
\\[-8pt]
\\[-8pt]
\nonumber
& \le& C(p)n\sup_{0\le i <\gD_n} \bigl(N^\circ_{n,i}
\bigr)^{p/2}.
\end{eqnarray}
Then, by Propositions \ref{p:scgw}(3), (5)
and~\ref{p:super}(3),
%
\begin{equation}
\sup_{0\le i <\gD_n} {N^\circ_{n,i}}/n^{1-\delta}
\pto0,
\end{equation}
and thus, with probability tending to 1 as ${n\to\infty}$,
%
\begin{equation}
\label{qw2} \sup_{0\le i <\gD_n} {N^\circ_{n,i}}
\le n^{1-\delta}.
\end{equation}
If $\tau_n$ is such that \eqref{qw2} holds then \eqref{qwe}, along with
Markov's inequality, implies
that for any $\eps>0$,
%
\begin{eqnarray}
\P \bigl(K> \eps n^{1/2} \mid\tau_n \bigr)& \le&
\eps^{-p} n^{-p/2} C(p) n ^{1+(1-\delta)p/2}
\nonumber
\\[-8pt]
\\[-8pt]
\nonumber
&=&
\eps^{-p}C(p) n ^{1-\delta p/2}\to0.
\end{eqnarray}
Hence, $\P(K>\eps n^{1/2})\to0$, as asserted.
\end{pf}

\begin{pf*}{Proof of Theorem~\ref{th:inv}}
To unify the treatment of the cases (\ref{A1}) and (\ref{A2}), we
define $p_0 = 1$ for the weights in (\ref{A2}). By Lemma~\ref{l:inv1},
it is sufficient to show that
%
\begin{equation}
\label{sup1} n^{-1/2} \sup_{0\leq x\leq N_n^\circ}
\biggl|L_n^\star \biggl(x\frac
{\Delta
_n}{N_n^\circ} \biggr)-L_n
(x )\biggr | \mathop{\xrightarrow }_{n\rightarrow
\infty
}^{\inp} 0.
\end{equation}
Note that $L_n$ is a linear interpolation of its values on the integers.
Using the triangle inequality (and Lemma~\ref{l:shift})
therefore allows us to restrict to integer
values of $x$ which we will write as $k$. Introduce the mapping $\pi_n\dvtx \{0,1,\ldots,N_n^\circ\} \rightarrow\{0,1,\ldots,\Delta_n\}$
defined as
follows; see Figure~\ref{Fpi}:
Let $\pi_n(0) = 0$ and $\pi_n(N_n^\circ) = \Delta_n$. If $v_{i}\in
\tau_{n,j}$ for $j = 1,\ldots,\Delta_n-1$ then $\pi_n(i) = j$. If
$v_0 <
v_{i} \leq s_0$ in the lexicographic ordering then $\pi_n(i) = 0$ and if
$v_i \in\tau_{n,0}$ with $v_{i}>s_0$ then $\pi_n(i) = \Delta_n$.
%
\begin{figure}

\includegraphics{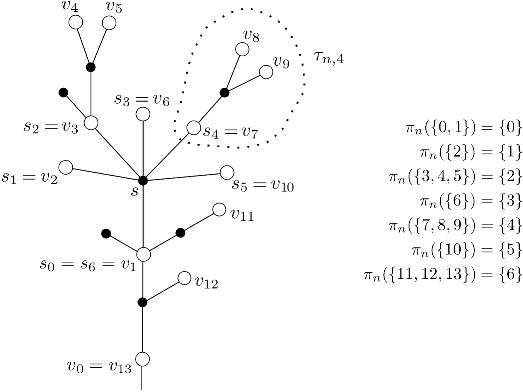}

\caption{An example of the mapping $\pi_n$.} \label{Fpi}
\end{figure}
By the triangle inequality,
%
\begin{eqnarray}
\label{pitri}&& n^{-1/2}\sup_{0\leq k\leq N_n^\circ} \biggl|L_n^\star
\biggl(k\frac{\Delta_n}{N_n^\circ} \biggr)-L_n (k )\biggr| \nonumber\\
&&\qquad\le n^{-1/2}
\sup_{0\leq k\leq N_n^\circ} \biggl|L_n^\star \biggl(k
\frac{\Delta_n}{N_n^\circ} \biggr) - L^\star _n\bigl(
\pi_n(k)\bigr)\biggr|
\\
&&\qquad\quad{}+ n^{-1/2}\sup_{0\leq k\leq N_n^\circ} \bigl|L_n(k) -
L^\star_n\bigl(\pi_n(k)\bigr)\bigr|.\nonumber
\end{eqnarray}
%
We begin by showing that the first term on the right-hand side of
(\ref{pitri}) converges to 0 in probability.
By Lemma~\ref{l:shift}, it suffices to show that
%
\begin{equation}
n^{-1}\sup_{0\leq k\leq N_n^\circ} \biggl|k\frac{\Delta_n}{N_n^\circ} -
\pi_n(k)\biggr| \pto0.
\end{equation}
%
We
have the estimate
%
\begin{equation}
\sum_{i=1}^{\pi_{n}(k)-1} N^\circ_{n,i}
\leq k \leq\sum_{i=0}^{\pi_{n}(k)}
N^\circ_{n,i}.
\end{equation}
Thus,
%
\begin{equation}
\Biggl|k - \sum_{i=1}^{\pi_{n}(k)} N^\circ_{n,i}\Biggr|
\le N^\circ_{n,0} + N^\circ_{n,\pi_n(k)}
\end{equation}
and hence, using
Propositions \ref{p:scgw}(3), (5)
and 
%
\begin{equation}
n^{-1}\sup_{0\le k\le N^\circ_n}\Biggl|k - \sum
_{i=1}^{\pi_{n}(k)} N^\circ_{n,i}\Biggr| \pto0.
\end{equation}

Furthermore, in view of
Propositions \ref{p:scgw}(1), (2)
and \ref{p:super}(1), (2),
$\gD_n/N^\circ_n\pto(1-\mean)/p_0$.
It
thus suffices to show that
%
\begin{equation}
\sup_{0\leq l \leq\Delta_n}n^{-1} \Biggl|\frac{1-\mean}{p_0}\sum
_{i=1}^{l} N^\circ_{n,i} - l\Biggr |
\mathop{\xrightarrow }_{n\rightarrow\infty
}^{\inp} 0,
\end{equation}
which indeed follows from Propositions \ref{p:scgw}(4)
and~\ref{p:super}(1).

Next, consider the second term on the right-hand side of (\ref{pitri}).
This is exactly the left-hand side in Lemma~\ref{L4}, and thus it to tends
to 0.
\end{pf*}

\section{Proof of Theorem \texorpdfstring{\protect\ref{th:main}}{1.1}} \label{s:proof}
We start by recalling standard results on the Gromov--Hausdorff
distance. A correspondence $\mathcal{R}$ between two metric spaces
$(E_1,d_1)$ and $(E_2,d_2)$ is a subset of $E_1 \times E_2$ such that
for every $x_1\in E_1$ there exists an $x_2\in E_2$ such that
$(x_1,x_2)\in\mathcal{R}$ and vice versa. Denote the set of all
correspondences between $E_1$ and $E_2$ by $\mathcal{C}(E_1,E_2)$. A
distortion of a correspondence is defined as
%
\begin{equation}
\operatorname{dis}(\mathcal{R}) = \sup\bigl\{ \bigl|d_1(x_1,y_1)-d_2(x_2,y_2)\bigr|
\dvtx (x_1,x_2),(y_1,y_2)\in
\mathcal{R}\bigr\}.
\end{equation}
The pointed Gromov--Hausdorff distance between $(E_1,d_1)$ and $(E_2,d_2)$
with marked points $\rho_1$ and $\rho_2$, respectively, can be conveniently
expressed as, see \cite{burago:2001}, Theorem~7.3.25 (for the nonpointed
version, the pointed version used here is similar)
%
\begin{equation}
\label{ghdis} d_{\mathrm{GH}}(E_1,E_2) =
\frac{1}{2}\inf_{\mathcal{R}\in\mathcal
{C}(E_1,E_2),(\rho_1,\rho_2)\in\mathcal{R}} \operatorname{dis}(\mathcal{R}).
\end{equation}

In the proof of Theorem~\ref{th:main}, we use similar ideas as in the
previous section. Let $M_n$ be a random planar map with a corresponding
mobile $\theta_n = (\tau_n,\ell_n)$. As before, we denote the white vertices
in $\theta_n$ by $v_{0},\ldots,v_{N^\circ_n}$ in lexicographical
order and
use the same notation for the corresponding white vertices in $M_n$. Also
define the vertex $s$ and its surrounding vertices
$s_{0},\ldots,s_{\Delta_N}$ as before. Denote by
$\theta^\star_n=(\tau_n^\star,\ell_n^\star)$ the mobile which is
obtained by
trimming $\theta_n$ such that it only consists of the black vertex $s$ and
its surrounding white vertices $s_i$, $0\leq i \leq N_n^\circ$, and keeping
the labels of these vertices the same as before.
We add a superscript $\star$ to the notation when we consider these
vertices as vertices in $\theta^\star_n$.
Take $s^\star_0$ to be the root of $\theta_n^\star$.
Note that if $L_n$ is the label process corresponding to $\theta_n$ then
$L^\star_n$, defined in Section~\ref{s:label}, is
the label process corresponding to $\theta_n^\star$.
By definition, it holds that $\ell_n^\star(s_{i}^\star) = \ell
_n(s_{i})$ for
all $0\leq i \leq\Delta_n$.
In general, the root $s^\star_0$ of $\theta_n^\star$ has a label different
from zero, but note that the BDG bijection still
works since it only depends on the
increments of the labels in the white contour sequence.
The planar map obtained from $\theta^\star_n$ is denoted by $M^\star
_n$, the
graph distance on $M^\star_n$ by $d_n^\star$ and the marked vertex by
$\rho^\star_n$.

The planar map $M^\star_n$ has a single black vertex and
has therefore a single face. Hence, it contains no cycles and is thus a
planar tree with $\Delta_n$ edges.
Given $\Delta_n$,
the map $M^\star_n$ is a uniformly distributed rooted planar tree and, given
$M^\star_n$, the marked vertex $\rho^\star_n$ is chosen uniformly at
random.
(Note that the root edge of $M_n^\star$ yields both a root vertex and an
ordering of the children of the root, and conversely; we may take the first
child to be the other endpoint of the root edge.)
Aldous~\cite{aldous:1993} proved that the contour function of such a random
rooted tree, after rescaling, converges in distribution to $\ex$,
which implies
convergence of the tree to $\mathcal{T}_\mathbf{e}$ in
Gromov--Hausdorff distance; see
\cite{legall:2005}, Theorem~2.5. Hence we obtain, including also the marked
vertex, the following.

\begin{theorem} \label{th:aldous}
For the weights \eqref{A1} and \eqref{A2}, the random planar map 
$ ((M^\star_n,\rho^\star_n ),(2(1-\mean)n)^{-1/2}d_n^\star)$
viewed as
an element of $\mathbb{M}^\ast$ converges in distribution toward
$ ((\mathcal{T}_\mathbf{e},\rho^\ast),\delta_\mathbf{e}
)$ where given
$\mathcal{T}_\mathbf{e}$, $\rho^*$ is a marked vertex chosen
uniformly at
random from $\mathcal{T}_\mathbf{e}$.
\end{theorem}

To complete the proof of Theorem~\ref{th:main}, we construct the following
correspondence between $((M_n,\rho_n),(2(1-\mean)n)^{-1/2}d_n)$ and
$((M^\star_n,\rho^\star_n),(2(1-\mean)n)^{-1/2}d_n^\star)$:
%
\begin{equation}
\mathcal{R}_n = \bigl\{\bigl(\rho_n,
\rho^\star_n\bigr)\bigr\} \cup\bigcup
_{i=0}^{N_n^\circ
-1} \bigl\{\bigl(v_{i},s_{\pi_n(i)}^\star
\bigr)\bigr\}
\end{equation}
with $\pi_n$ the same as in the proof of Theorem~\ref{th:inv}. We then
show that the distortion of this correspondence converges to zero in
probability. Recall the definition of $\tau_{n,i}$ in Section~\ref{s:anotherb}. We have the following estimate.

\begin{lemma} \label{l:dis}
For any mobile $\theta_n = (\tau_n,\ell_n)$ it holds that
%
\begin{equation}
\label{dis} \operatorname{dis}(\mathcal{R}_n) \leq\bigl(2(1-\mean)n
\bigr)^{-1/2} \Bigl(14 \sup_{0
\leq i \leq\Delta_n-1} \sup
_{v \in\tau_{n,i}}\bigl|\ell_{n}(v)-\ell _{n}(s_i)\bigr|+4
\Bigr).
\end{equation}
\end{lemma}

By Lemma~\ref{L4}, the right-hand side tends to 0 in probability,
which along with Theorem~\ref{th:aldous} completes the proof of Theorem~\ref{th:main}. We conclude by proving Lemma~\ref{l:dis}.

\begin{pf*}{Proof of Lemma~\ref{l:dis}}
Let $(x,x^\star),(y,y^\star)\in\mathcal{R}_n$. Write
\[
K = \sup_{0 \leq i \leq\Delta_n-1} \sup_{v \in\tau_{n,i}}\bigl|\ell
_{n}(v)-\ell_{n}(s_i)\bigr|.
\]
When we refer to ancestral relations in $M_n^\star$ we use $\rho
_n^\star
$ as the reference point, that is, we say that $y$ is an ancestor of
$x$ in $M_n^\star$ if $x\neq y$ and the unique geodesic from $x$ to
$\rho_n^\star$ contains $y$. Consider separately the following three cases:
\begin{longlist}[(1)]
\item[(1)]
$x^\star= y^\star$.
\item[(2)]
$y^\star$ is an ancestor of $x^\star$ in $M_n^\star$,
or conversely.
\item[(3)]
$x^\star\neq y^\star$ and neither is an ancestor of the other.
\end{longlist}
Begin by studying case (1). The case $x^\star= y^\star= \rho
_n^\star$ is trivial so we consider $x^\star\neq\rho_n^\star$. We can
then write $x^\star= y^\star= s_{i}^\star$ for some $i$ which will be
fixed in this part of the proof. Let $\lambda_0$ denote the minimum
label in $\tau_{n,i}$. For a vertex $v\in V(M_n)$ define the successor
geodesic $\gamma(v)$ from $v$ to $\rho_n$ by $(v,\sigma(v),\sigma
\circ
\sigma(v),\ldots,\rho_n)$ with $\sigma$ defined in (\ref{succ}). Then
there is a vertex $w$ with label $\ell_n(w) = \lambda_0-1$ such that
$\gamma(x)$ and $\gamma(y)$ contain $w$. Therefore, it follows from
(\ref{distances}) and the definition of $K$ that
\begin{eqnarray*}
d_n(x,w) &=& \ell_n(x) - \lambda_0+1 \leq2K+1\quad
\mbox{and}
\\
d_n(y,w) &=& \ell_n(y) - \lambda_0+1
\leq2K+1.
\end{eqnarray*}
Thus, by the triangle inequality,
%
\begin{equation}
\label{estin1} \bigl|d_n(x,y)-d^\star_n
\bigl(x^\star,y^\star\bigr)\bigr| = d_n(x,y) \leq
d_n(x,w) + d_n(y,w) \leq4K+2.
\end{equation}

Next, consider case (2) and assume that $y^\star$ is an ancestor
of $x^\star$. First, assume that $y^\star\neq\rho_n^\star$. Then
there are unique $i$ and $j$ such that $x^\star= s^\star_{i}$,
$y^\star= s^\star_{j}$ and without loss of generality we assume that
$i < j$ (otherwise we shift the indices $i$ and $j$ modulo $\Delta_n$).
In this part, $i$ and $j$ are fixed. It holds that
%
\begin{equation}
\label{complabels1} \ell_n(s_{m}) > \ell_n(s_{j})
\end{equation}
for all $m$ obeying $i \leq m < j$. Let $\gamma_i$ be a successor
geodesic from $s_i$ to $\rho_n$ and let $\gamma_j$ be a successor
geodesic in $M_n$ from $s_j$ to $\rho_n$. We will show that the
distance between $\gamma_i$ and $s_{j}$ is small (in terms of $K$). Define
%
\begin{equation}
\label{lambda1def} \lambda_1 = \min\bigl\{\ell_n(v_{m})
\dvtx i \leq\pi_n(m) < j\bigr\}.
\end{equation}
It clearly holds that
%
\begin{equation}
\lambda_1 \leq\ell_n(s_{j-1})\le
\ell_n(s_{j})+1.
\end{equation}
Let $l$ be an index for which the minimum in (\ref{lambda1def}) is
attained, that is, such that $\ell_n(v_l) = \lambda_1$ and $i \leq
\pi
_n(l) < j$. Then, by (\ref{complabels1}),
%
\begin{equation}
\label{lam1est} \lambda_1 = \ell_n(v_{l})
\geq\ell_n(s_{\pi_n(l)}) - K \geq\ell _n(s_{j})+1-K.
\end{equation}
Now, $\gamma_i$ and $\gamma_j$ intersect for the first time at a vertex
with label $\lambda_1-1$, and call this vertex $z$; see Figure~\ref{f:case2}. Then by (\ref{distances}) and (\ref{lam1est})
%
\begin{equation}
\label{zest} d_n(z,s_{j}) = \ell_n(s_{j})-
\ell_n(z) \leq K.
\end{equation}
Furthermore, with the same argument leading to \eqref{estin1}
%
\begin{equation}
\label{xyest} d_n(x,s_{i}) \leq3K+2 \quad\mbox{and}\quad
d_n(y,s_{j}) \leq3K+2.
\end{equation}

\begin{figure}[t]

\includegraphics{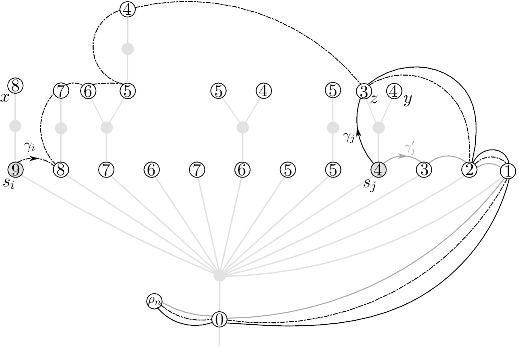}

\caption{Illustration of the setup in part (2) of the
proof with
$\ell_n(s_{j}) = 4$ and $\lambda_1 = 4$. A planar mobile is shown with
the edges and the black vertices colored light gray. The geodesic
$\gamma_i$ is black with dots and dashes and $\gamma_j$ is black and
solid. Since they are successor geodesics, they intersect at the vertex
$z$ with label $\lambda_1-1=3$. Another geodesic $\gamma_j'$ from
$s_{j}$ to $\rho_n$ (not a successor geodesic) is shown in dark gray
and it does not intersect $\gamma_i$ at a vertex with label $\lambda
_1-1$.} \label{f:case2}
\end{figure}
Finally, we get by repeatedly using the triangle inequality along with
(\ref{distances}), (\ref{zest}) and (\ref{xyest})
%
\begin{eqnarray}
\label{upper2}
\nonumber
\bigl|d_n(x,y)-d_n^\star
\bigl(x^\star,y^\star\bigr)\bigr| &\leq&\bigl |d_n(x,y)-d_n(s_{i},s_{j})\bigr|+
\bigl|d_n(s_{i},s_{j})-d_n(z,s_{i})\bigr|
\\
\nonumber
&&{} +\bigl |d_n(z,s_{i})-d_n^\star
\bigl(x^\star,y^\star\bigr)\bigr|
\\
&\leq& d_n(x,s_{i})+d_n(y,s_{j})
+ d_n(z,s_{j})
\\
\nonumber
&&{} +\bigl |\ell_n(s_{i})-\ell_n(z)-
\ell_n^\star\bigl(s_{i}^\star\bigr)+\ell
^\star _n\bigl(s^\star_{j}\bigr)\bigr|
\\
&\leq& 7K+4 + \bigl|\ell_n(s_{j})-\ell_n(z)\bigr|
\leq8K+4.\nonumber
\end{eqnarray}
The case $y^\star= \rho_n^\star$ is treated in a simpler way leading
to a
similar upper bound as in (\ref{upper2}); we omit the details.

Finally, consider case (3) (see Figure \ref{f:case3} for an illustration).
We keep writing $x^\star=
s_{i}^\star
$ and
$y^\star= s_{j}^\star$. Denote the common ancestor of $x^\star$ and
$y^\star$ having the largest label by $z^\star$. Assume that $z^\star
\neq
\rho_n^\star$ and write $z^\star= s_{k}^\star$; the case $z^\star=
\rho_n^\star$ is treated in a similar but simpler way.
We may assume without loss of
generality that $i < j < k$ (by shifting the indices modulo $\Delta_n$ and
possibly renaming $x$ and $y$). In this part, $i$, $j$ and $k$ are
fixed. Define the geodesics $\gamma_i$ and $\gamma_j$ as in case (2)
and let $\gamma_k$ be the successor geodesic from $s_k$ to
$\rho_n$. Furthermore,
let $\gamma_{ij}$ be a geodesic directed from $s_{i}$ to $s_{j}$ in
$M_n$. Since $s_{k}^\star$ is an ancestor of both $x^\star$ and
$y^\star
$, it
follows from (\ref{zest}) that there is a vertex $z_i$ in $\gamma_i$
and a
vertex $z_j$ in $\gamma_j$ such that
%
\begin{equation}
\label{zizj} d_n(z_m,s_{k}) =
\ell_n(s_{k})-\ell_n(z_m) \leq K\qquad
\mbox{for } m = i,j.
\end{equation}
Moreover,
%
\begin{equation}
\label{complabels2} \ell_n(s_{m}) > \ell_n(s_{k})
\end{equation}
for all $m$ obeying $i \leq m < k$. We now show that $\gamma_{ij}$ is
also close to $s_{k}$. Define
%
\begin{equation}
\label{lambda2} \lambda_2 = \min\bigl\{\ell_n(v_{m})
\dvtx v_{m} \in\gamma_{ij}, i \leq\pi _n(m) <
k\bigr\},
\end{equation}

\begin{figure}

\includegraphics{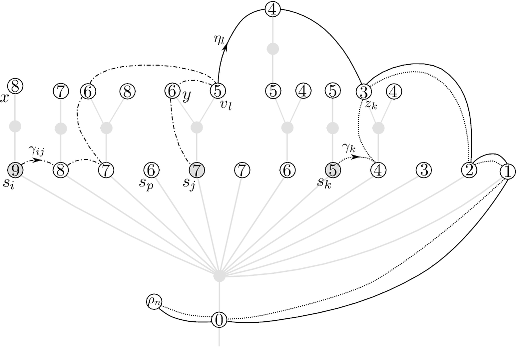}

\caption{Illustration of the setup in part (3) of the
proof with
$\ell_n(s_{k}) = 5$, $\lambda_2=5$ and $\lambda_3 = 4$. A~planar mobile
is shown with the edges and the black vertices colored light gray. The
geodesic $\gamma_{ij}$ is dotted and dashed, $\eta_l$ is solid and
$\gamma_k$ is dotted. Since $\eta_l$ and $\gamma_k$ are successor
geodesics, they intersect at the vertex $z_{k}$ with label $\lambda
_3-1=3$.} \label{f:case3}
\end{figure}

\noindent where by $v_{m} \in\gamma_{ij}$ we mean that $v_{m}$ is visited by
$\gamma_{ij}$.
Condition (3) guarantees that there is an index, say $p$, such
that $i \leq p < j$ and $\ell_n(s_p) = \ell_n(s_k)+1$. Let $q$ be the
first time at which $s_p < \gamma_{ij}(q) \leq s_k$ in the
lexicographic order on $\tau_n$. Then $q$ is well defined since
$\gamma
_{ij}$ ends at $s_j$.
If $\ell_n(\gamma_{ij}(q)) = \ell_n(\gamma_{ij}(q-1))+1$ then by the
properties of the BDG bijection $\ell_n(\gamma_{ij}(q))\leq\ell
_n(s_k)$. On the other hand, if $\ell_n(\gamma_{ij}(q)) = \ell
_n(\gamma
_{ij}(q-1))-1$ then by the same arguments $\ell_n(\gamma_{ij}(q-1))
\leq\ell_n(s_p)$ which again yields $\ell_n(\gamma_{ij}(q))\leq
\ell
_n(s_k)$. We have thus established that
%
\begin{equation}
\label{labelclaim} \lambda_2 \leq\ell_n(s_{k}).
\end{equation}
Let $l$ be an index for which the minimum in (\ref{lambda2}) is
attained, that is, such that $v_{l} \in\gamma_{ij}$, $i\leq\pi_n(l) <
k$ and $\ell_n(v_{l}) = \lambda_2$. By (\ref{complabels2}),
%
\begin{equation}
\label{lam2est} \lambda_2 = \ell_n(v_{l}) \geq
\ell_n(s_{\pi_n(l)})-K \geq\ell _n(s_{k})+1-K.
\end{equation}
Denote the successor geodesic from $v_{l}$ to $\rho_n$ by $\eta_l$.
Next, define
%
\begin{equation}
\lambda_3 = \min\bigl\{\ell_n(v_{m}) \dvtx
\pi_n(l) \leq\pi_n(m) < k\bigr\}.
\end{equation}
Now, $\eta_l$ and $\gamma_k$ intersect for the first time at a vertex having
label $\lambda_3 - 1$, and call this vertex $z_{k}$. With same argument
as in
\eqref{lam1est},
%
\begin{equation}
\label{l3ub} \lambda_3 \geq\ell_n(s_{k}) +
1 - K
\end{equation}
and this yields, along with (\ref{distances}) and (\ref{labelclaim})
%
\begin{equation}
\label{l2l3diff} d_n(v_{l},z_k) =
\ell_n(v_{l}) - \ell_n(z_k) =
\lambda_2-\lambda_3 + 1\leq K.
\end{equation}
Also, by (\ref{distances}) and (\ref{l3ub})
%
\begin{equation}
\label{dsnkzk} d_n(s_{k},z_k) =
\ell_n(s_{k})-\ell_n(z_k) \leq K.
\end{equation}
Using the triangle inequality along with (\ref{l2l3diff}) and (\ref
{dsnkzk}), we get
%
\begin{eqnarray}
\label{vnl} d_n(v_{\ell},s_{k}) &\leq&
d_n(v_{l},z_k)+d_n(z_k,s_{k})
\leq2K.
\end{eqnarray}

Finally, we obtain by using the triangle inequality, (\ref{distances}),
(\ref{xyest}), (\ref{zizj}) and~(\ref{vnl})
\begin{eqnarray*}
&& \bigl|d_n(x,y)-d_n^\star\bigl(x^\star,y^\star
\bigr)\bigr| \\
&&\qquad\leq \bigl|d_n(x,y)-d_n(s_{i},s_{j})\bigr|
\\
&&\qquad\quad{} +\bigl |d_n(s_{i},v_{l})-d_n(s_{i},s_k)\bigr|
+ \bigl|d_n(s_{i},s_k)-d_n(s_i,z_i)\bigr|
\\
&&\qquad\quad{} + \bigl|d_n(s_{j},v_{l})-d_n(s_{j},s_k)\bigr|
+\bigl |d_n(s_{j},s_k)-d_n(s_j,z_j)\bigr|
\\
&&\qquad\quad{} + \bigl|d_n(s_i,z_i)+d_n(s_j,z_j)-d_n^\star
\bigl(x^\star,y^\star\bigr)\bigr|
\\
&&\qquad \leq d_n(x,s_i)+d_n(y,s_j)+2d_n(s_k,v_l)
+ d_n(z_i,s_k) + d_n(z_j,s_k)
\\
&&\qquad\quad{} + \bigl|\ell_n(s_i)-\ell_n(z_i)+
\ell_n(s_j)-\ell_n(z_j)-
\ell^\star _n\bigl(s_i^\star\bigr)-
\ell^\star_n\bigl(s_j^\star\bigr)+2
\ell^\star_n\bigl(s_k^\star\bigr)\bigr|
\\
&&\qquad \leq 12K+4 + \bigl|\ell_n(s_k)-\ell_n(z_i)\bigr|+\bigl|
\ell_n(s_k)-\ell_n(z_j)\bigr|
\leq14K+4. 
\end{eqnarray*}
\upqed\end{pf*}

\section{Conclusions} \label{s:conc}
We have shown that the random planar maps defined by the weights (\ref{A1})
and (\ref{A2}) converge to Aldous' Brownian tree. It is interesting to note
that there does not seem to be a nontrivial scaling limit of the
corresponding simply generated trees; see \cite{kortchemski:2012},
Theorem~6,
and thus the labels play a crucial role in obtaining a scaling limit
for the
random maps.

One can also study the so-called local limit of the planar maps $M_n$
under consideration in this paper. The limit, when it exists, is an
infinite graph $M$ and convergence toward $M$ roughly means that one
considers all finite neighborhoods of faces around the root edge and
shows that the probability that they appear in the maps $M_n$
converges, as $n\rightarrow\infty$, to the probability that they
appear in $M$. Angel and Schramm \cite{angel:2003} studied local
convergence in the case of uniformly distributed triangulations (all
faces have degree 3) and later Durhuus and Chassaing \cite
{chassaing:2006} and Krikun \cite{krikun:2005} studied the case of
uniformly distributed quadrangulations (all faces have degree 4).
Recently, there have been several new results on the local limit of
uniform quadrangulations concerning, for example, properties of
infinite geodesics~\cite{curien:2012a}, random walks \cite
{benjamini:2012} and quadrangulations with a boundary \cite
{curien:2012b}. In a forthcoming paper \cite{bjornberg:2013}, it is shown that the local limit $M$ of the maps $M_n$
distributed by \eqref{mumeasure} exists for all choices of weights
$q_i$. The proof involves using the bijection $\mathcal{G}_n$
introduced in the current paper along with theorems on local
convergence of simply generated trees which we now briefly review.

The local limit of the simply generated trees corresponding to the weights
(\ref{A1}) and (\ref{A2}) was established in \cite{jonsson:2011}
(with an
asymptotically constant slowly varying function) and \cite{janson:2011},
respectively. Later it was established in full generality [covering cases
(\ref{cond:A1}) and (\ref{cond:A2})] in \cite{janson:2012}. In case
(\ref{cond:A2}), the local limit is deterministic and equals the infinite
star, that is, the root has a single neighbor of infinite degree and
all its
neighbors are leaves. Therefore, the local limit $M$ of the
corresponding planar maps is simply the infinite uniform planar tree. In
case (\ref{cond:A1}), the local limit of the trees is more
complicated. It
still has a unique vertex of infinite degree but the outgrowths from this
vertex are now i.i.d. subcritical Galton--Watson trees. Therefore, the
local limit $M$ of the corresponding maps is not a tree. However, since
subcritical Galton--Watson trees tend to be small, it is interesting
to see how different $M$ is from the uniform tree. It is, for example,
interesting to study
properties of random walks on $M$ since random walks are sensitive to
the presence of loops. In \cite{bjornberg:2013}, it is shown
(under some moment conditions on the weights $w_i$) that the spectral
dimension of $M$, a number which characterizes the rate of decay of the
return probability of the random walk, equals $4/3$ which is indeed the same
value as for the uniform infinite planar tree.

A natural question to ask is how universal our results are, that is, is it
enough to pose the conditions (\ref{cond:A1}) or (\ref{cond:A2}) in the
\hyperref[s:intro]{Introduction} or does one have to go to special cases? It is shown in
\cite{janson:2012}, Examples 19.37--19.39, that by choosing irregular
weights, still satisfying (\ref{cond:A1}) or (\ref{cond:A2}), the
corresponding simply generated trees with $n$ edges can have more than one
vertex with a degree of the order of $n$; it is even possible that the large
vertices have degrees $o(n)$ and that their number goes to infinity as
$n\rightarrow\infty$ (at least along subsequences).
In the case when there are two vertices with degrees
of the order of $n$, it is plausible that the planar maps have a scaling
limit which is roughly the Brownian tree with two points identified, forming
a second macroscopic face. The more there is of large vertices in the simply
generated trees the more faces should appear in the scaling limit of the
maps. Thus, we conjecture that the Brownian tree only appears in special
cases of~(\ref{cond:A1}) and (\ref{cond:A2}).
We consider, for simplicity, only one simple example
(similar to~\cite{janson:2012}, Example~19.38)
illustrating this.

\begin{example}
Let $(w_i)_{i\ge0}$ be a weight sequence such that
$w_i=0$ unless $i\in\{0,3^j\dvtx j\ge0\}$.
Further, let $w_0=1$ and let $w_{3^j}$ increase so rapidly that
\eqref{cond:A2} holds, and moreover,
with probability tending to 1 as $k\to\infty$,
if $n=3^k$, then
the simply generated random
tree $\tau_n$ with the distribution $\nu_n$
given by \eqref{nu}
is a star, while if $n=2\cdot3^k$, then $\tau_n$ has two vertices of
outdegree $n/2=3^k$ (and all other vertices are leaves).

For the subsequence $n=3^k$, we then obtain the same results as above in
the case \eqref{A2}.

For the subsequence $n=2\cdot3^k$,
the corresponding coloured tree distributed by $\tilde\nu$ has (with
probability tending to 1) two black vertices of degrees $n/2$ connected
by a
single white vertex $\vv$ of degree 2, and each of them joined to
$n/2-1$ white
leaves. For each choice of labels $\ell_n$, the corresponding map
$M_n$ thus
has two faces. The label processes around each black vertex converge to
independent Brownian bridges, which together with the random choice of root
implies that, in analogy to Theorem~\ref{th:inv},
%
\begin{equation}
\biggl(\frac{1}{\sqrt{n}}L_n\bigl(t N^\circ_n
\bigr) \biggr)_{0 \leq t \leq1} \mathop{\xrightarrow}_{n\rightarrow\infty}^{\ind}
\bigl(\mathbf {h}(t) \bigr)_{0
\leq t \leq1},
\end{equation}
where, for two Brownian bridges $\br_1,\br_2$ and $U$ uniformly
distributed on $[0,1]$, all independent,
%
\begin{equation}
\qquad\mathbf h(t)= %
\cases{\br_1(2t+U)-\br_1(U),
& \quad $0\le t\le(1-U)/2,$ \vspace*{2pt}
\cr
\br_2(2t+U-1)-
\br_1(U), &\quad $(1-U)/2\le t\le1-U/2,$ \vspace*{2pt}
\cr
\br_1(2t+U-2)-\br_1(U),&\quad  $1-U/2\le t\le1.$} %
\end{equation}
Moreover, the label process visits $\vv$, the unique white vertex of
degree 2,
twice. If we split this vertex into two, the corresponding map will be a
tree, which after normalization converges in distribution in the
Gromov--Hausdorff metric to a random real tree $\cT_{\h'}$, where $\h'$
is the
random function $h$ above shifted to its minimum and with the minimum
subtracted, so $\h'\ge0$ and $\h'(0)=0$.
We may by the Skorohod representation theorem assume that the
label processes converge a.s. Then the
random maps with $\vv$ split converge to $\cT_{\h'}$ a.s. in the
Gromov--Hausdorff metric,
with the two halves of $\vv$ corresponding to two different points in
$\cT_{\h'}$ [the points given by $t=(1-U)/2$ and $t=1-U/2$], and it follows
by combining the two parts of $\vv$ again, that the random maps $M_n$
converge to a limit that equals $\cT_{\h'}$ with these two points
identified. Note that this creates a cycle, so the limit is no longer a
tree.
(As a topological space, it is of the same homotopy type as a circle.)
\end{example}

\begin{appendix}\label{app}
\section*{Appendix: More on Galton--Watson trees}\label{AGW}
Marckert and Miermont \cite{marckert:2007} gave a description of the
distribution $\tilde\nu$ in \eqref{tnu} as a conditioned
two-type Galton--Watson tree, while we have used the bijection $\cG_n$ in
Section~\ref{s:anotherb} to obtain a simply generated tree (which in many
cases is
a conditioned Galton--Watson tree), with a single type only.
In this appendix, we give some further comments on the relation between
these two approaches.

Consider arbitrary weights $q_i\ge0$, $i\ge1$, assuming first only that
$q_i>0$ for some $i>1$ (to avoid trivialities), and define $w_i$ by
\eqref{eq:defp} (and $w_0=1$)
and their generating function $g(x)$ by \eqref{generating}.
Marckert and Miermont \cite{marckert:2007} define another generating
function $f(x)$ (denoted $f_{\mathbf q}(x)$ in \cite{marckert:2007}) by
%
\setcounter{equation}{0}
\begin{equation}
f(x)=\sum_{k=0}^\infty w_{k+1}x^k;
\end{equation}
thus
%
\begin{equation}
\label{gf} g(x)=1+xf(x).
\end{equation}

We have seen in Sections~\ref{s:intro} and \ref{s:anotherb} that a random
planar map in $\cM_n^*$ with Bolzmann weights \eqref{wm} corresponds
to a
random mobile $(\tau_n,\ell_n)$
(and a sign $\varepsilon$ that we ignore here), and that $\tau_n$
corresponds by the bijection $\cG_n$ to a random tree $\tau_n'$
that has the distribution of a simply generated tree with $|\tau_n'|=n$
edges, defined by the weights $(w_i)_{i\ge0}$, cf. \eqref{nu} [and note
that $\deg(v)-1$ is the outdegree, i.e., the number of children of
$v$; see
Section~\ref{s:bdg}].

We consider first trees with unrestricted number of edges.
We give a planar tree~$\tau$ the weight
%
\begin{equation}
\label{wtau} w(\tau)=\prod_{v\in V(\tau)}w_{\deg(v)-1}.
\end{equation}
The generating function
%
\begin{equation}
\label{G} G(x)=\sum_{\tau} x^{|\tau|+1} w(
\tau) 
\end{equation}
summing over all planar trees,
satisfies the well-known equation \cite{Otter}
%
\begin{equation}
G(x)=x g\bigl(G(x)\bigr).
\end{equation}
In particular, the total weight $Z=\sum_\tau w(\tau)=G(1)$ is finite
if and
only if the
equation
%
\begin{equation}
\label{G1} z=g(z)
\end{equation}
has a solution $z\in(0,\infty)$, and then $Z$ is the smallest positive
solution to \eqref{G1}.
Using \eqref{gf}, we can write \eqref{G1} as
$z=1+zf(z)$, or
%
\begin{equation}
\label{cl} f(z)=1-1/z,
\end{equation}
the form of the equation used in \cite{marckert:2007}.

If $Z=G(1)<\infty$ (such weights $q_i$ are called \emph{admissible} in
\cite{marckert:2007}), define
%
\begin{equation}
\label{hp} \hp_i=w_iZ^{i-1}.
\end{equation}
Then, by \eqref{G1},
%
\begin{equation}
\sum_{i=0}^\infty\hp_i =
Z^{-1}g(Z)=1,
\end{equation}
so $(\hp_i)_{i\ge0}$ is a probability distribution on $\{0,1,\ldots\}$.
Let $\tau'$ be a random Galton--Watson tree with this offspring distribution.
Then the probability of a particular realization $\tau'$ is
%
\begin{equation}
\label{gwp} \prod_{v\in V(\tau')} \hp_{\deg(v)-1} =
Z^{\sum_v(\deg(v)-2)}\prod_{v\in V(\tau')} w_{\deg(v)-1} =
Z^{-1}w\bigl(\tau'\bigr), 
\end{equation}
recalling that the number of vertices in $\tau'$
is $|\tau'|+1$ and that
$\sum_v\deg(v)=2|\tau'|+1$ since we count an extra half-edge at the root.
Hence, the distribution of the Galton--Watson tree $\tau'$ equals the
distribution given by the 
weights $w(\tau)$ on the set of all
planar trees.

\begin{remarkk}\label{Rlika}
The distribution $(p_i)_{i\ge0}$ defined by \eqref{hp} is not the
same as the
$(p_i)_{i\ge0}$ used in Section~\ref{s:anotherb}, so they define
different Galton--Watson
trees $\tau'$; however, they yield the same distribution $\nu_n$
when conditioned on a fixed size $n$ of the tree.
\end{remarkk}

Since $Z=\sum_\tau w(\tau)$, the sum of the probabilities \eqref
{gwp} over
all (finite) $\tau$ is~1; thus the Galton--Watson tree $\tau'$ is
a.s. finite,
which means that the offspring distribution has mean $\le1$, that is, the
Galton--Watson tree is subcritical or critical.
Conversely, we can obtain any subcritical or critical probability
distribution $(p_i)_{i\ge0}$ by taking $w_i=p_0^{i-1}p_i$; then
$w_0=1$ and
$Z=p_0^{-1}$.
(If we do not insist on $w_0=1$, we can simply take $w_i=p_i$.)

The offspring distribution \eqref{hp} has probability generating function
%
\begin{equation}
g_{\mathbf p}(x)= \sum_{i=0}^\infty
\hp_i x^i = Z^{-1}g(Zx)
\end{equation}
and thus mean
%
\begin{equation}
g_{\mathbf p}'(1)=g'(Z), 
\end{equation}
which by \eqref{gf} and \eqref{cl} can be written as
%
\begin{equation}
g_{\mathbf p}'(1)=f(Z)+Zf'(Z)=1+
\bigl(Z^2f'(Z)-1\bigr) /Z.
\end{equation}
Hence, the Galton--Watson tree is critical if and only if $g'(Z)=1$ or,
equivalently, $Z^2f'(Z)=1$ (the form used in \cite{marckert:2007}).
Moreover, the variance of the offspring distribution is
%
\begin{equation}
\gs^2= g_{\mathbf p}''(1)+g_{\mathbf p}'(1)-
\bigl(g_{\mathbf p}'(1)\bigr)^2 =Zg''(Z)+g'(Z)
\bigl(1-g'(Z)\bigr),
\end{equation}
which in the critical case $g'(Z)=1$ can be written by \eqref{gf} as
%
\begin{equation}
\gs^2 = Zg''(Z)=Z \bigl(Zf''(Z)+2f'(Z)
\bigr) = \bigl(Z^3f''(Z)+2 \bigr)/Z,
\end{equation}
which in the notation of \cite{marckert:2007} is $\rho_{\mathbf
q}/Z_{\mathbf q}$.

The two-type Galton--Watson tree defined by Marckert and Miermont
\cite{marckert:2007},
which we denote by $\tau$,
has a white root;
a white vertex has only black children, and the number of them has the
geometric distribution
$\Ge(\hp_0)= ((1-Z^{-1})^iZ^{-1} )_{i\ge0}$;
a black vertex has only white children, and the number of them has the
distribution
$ (\hp_{i+1}/(1-\hp_0) )_{i\ge0}
=
(\hp_{i+1}/(1-Z^{-1}) )_{i\ge0}$,
that is, the conditional distribution of $(\hxi-1\mid\hxi>0)$
if $\hxi$ has the distribution $(\hp_i)_{i\ge0}$.
Thus, the offspring distribution for the black vertices
has the probability generating function
%
\begin{equation}
\sum_{i=0}^\infty\frac{\hp_{i+1}}{1-Z^{-1}}
x^i = \sum_{i=0}^\infty
\frac{w_{i+1}Z^ix^i}{1-Z^{-1}} =\frac{f(Zx)}{1-Z^{-1}} =\frac{g(Zx)-1}{(Z-1)x}.
\end{equation}
A simple calculation
(which essentially is \cite{marckert:2007}, Proposition~7) shows that the
bijection in Section~\ref{s:anotherb}
maps this two-type Galton--Watson tree $\tau$ to the standard (single type)
Galton--Watson tree $\tau'$ with offspring distribution \eqref{hp}.
This can also be seen from the construction of the bijection;
see Figure~\ref{f:treetotree}. In particular, note that the children of the root in
$\tau$ are the vertices in the rightmost path from the root in~$\tau'$,
excluding its final leaf (and similarly for the children of other white
vertices); this explains why the offspring distribution for a white vertex
is geometric, since the length of the rightmost path in $\tau'$ obviously
has a geometric distribution.

Restricting to trees with $n$ edges (and thus $n+1$ vertices) we see,
by Remark~\ref{Rlika}, that
the tree $\tau_n'$ in Section~\ref{s:anotherb} with distribution $\nu_n$
can be seen as $\tau'$ conditioned on $|\tau'|=n$, and thus the corresponding
tree $\tau_n=\cG_n^{-1}(\tau'_n)$ has the same distribution as $\tau$
conditioned on $|\tau|=n$.

Although the Galton--Watson tree $\tau'$ is simpler than the two-type tree
$\tau$, the latter is more convenient for some purposes.
For example, when considering the white vertices, as we do in parts of
Section~\ref{s:anotherb}, it is immediate (by considering each second
generation)
that the number of white vertices in $\tau'$ is distributed as the total
progeny (number of vertices) in a Galton--Watson tree with offspring
distribution
%
\begin{equation}
\label{xi0} \xi^{(0)}=\sum_{j=1}^{\zeta}
\xi^*_j,
\end{equation}
where $\zeta\sim\Ge(\hp_0)=\Ge(1-Z^{-1})$ and
$\xi^*_j=(\hxi_j-1\mid\hxi_j>0)$ are
independent of each other and of $\zeta$, and each $\hxi_j$ has the
distribution $(\hp_i)_{i\ge0}$.
We have, letting $\kappa=\E\xi=\sum_i ip_i \le1$,
%
\begin{equation}
\label{exix} \E\xi^*_i = \E(\xi_i\mid
\xi_i>0) -1 =\frac{\E\xi_i}{1-p_0}-1=\frac{\mean+p_0-1}{1-p_0}
\end{equation}
and
%
\begin{equation}
\label{er} \quad\E\xi^{(0)}=\E\zeta\E\xi^*_1 =
\frac{1-p_0}{p_0}\frac{\mean
+p_0-1}{1-p_0} =\frac{\mean+p_0-1}{p_0} =1-\frac{1-\mean}{p_0}.
\end{equation}
Furthermore,
it is easy to see that $\xi^{(0)}$ has the probability generating function
%
\begin{equation}
\E x^{\xi^{(0)}} 
= \frac{\hp_0}{1-\sum_{k=1}^\infty\hp_k x^{k-1}}.
\end{equation}
Note that $\E\xi^{(0)}<1$ when $\E\xi<1$, which says that the white tree
consisting of each second generation in $\tau$ is subcritical if and
only if
$\tau'$ (or $\tau$) is.

Translated to $\tau'$, this shows immediately that the number of
leaves of
the Galton--Watson tree $\tau'$ with offspring distribution $\hxi$
is distributed as the total progeny of a Galton--Watson process with
offspring distribution $\xi^{(0)}$. In fact, this was shown
by Minami \cite{minami:2005}; one version of his argument is the following.
Given a tree $\tau$, we partition its vertex set into \emph{twigs} as
follows:
Take the vertices in lexicographic order and stop each time we reach a
leaf,
that is, the first twig consists of the root and all vertices up to, and
including, the first leaf; the second twig starts at the next vertex and
ends at the next leaf, and so on.
Thus, each twig ends with a leaf, and the number of twigs equals the number
of leaves.
If we start with a random Galton--Watson tree $\tau'$ with offspring
distribution $(\hp_i)$, the size of each twig has a geometric
distribution $1+\zeta$ with $\zeta\sim\Ge(\hp_0)$ as above.
Moreover, each nonleaf in the twig has further offspring distributed as
$\xi^*$; hence, if we contract each twig to a single vertex,
we obtain a new random Galton--Watson tree with
offspring distributed
as $\xi^{(0)}$; the number of vertices in this tree equals the number
of twigs
in $\tau'$, and thus the number of leaves in $\tau'$.

In fact, these two arguments are essentially the same; if we use instead
the reverse lexicographic order when defining the twigs, it is easy to see
that each twig in $\tau'$ correspond to a white vertex and its (black)
children in $\tau$.

We use this representation to verify the tail estimate \eqref{n0tail}.

\begin{lemmaa}\label{LN0}
Let $N^{(0)}$ be the number of leaves in a Galton--Watson tree with
offspring distribution $(p_i)_{i\ge0}$ satisfying $\kappa<1$ and
\eqref{probox} for some slowly
varying function $\LL(i)$.
Then, as ${n\to\infty}$,
%
\begin{equation}
\P\bigl(N^{(0)}=n\bigr)\sim c \LL(n) n^{-\gb},
\end{equation}
with $c={p_0}^{\gb-1} (1-\mean)^{-\gb}$.
\end{lemmaa}

\begin{pf}
We have seen that $N^{(0)}$ is distributed as the number of vertices in a
Galton--Watson tree with offspring distribution \eqref{xi0}.
By \eqref{jk} applied to a sequence $\xi_j^{(0)}$ of independent
copies of $\xi^{(0)}$,
%
\begin{equation}
\label{aw1} \P\bigl(N^{(0)}=n\bigr)=\frac{1}n\P\bigl(
\So_n=n-1\bigr),
\end{equation}
where
%
\begin{equation}
\label{aw2} \So_n = \sum_{j=1}^n
\xi^{(0)}_j \stackrel{\mathrm{d}} {=}\sum
_{i=1}^{X_n}\xi^*_i,
\end{equation}
where $X_n=\sum_{j=1}^n\zeta_j$ with $\zeta_j\sim\Ge(p_0)$ independent
of each other and of $\{\xi^*_i\}$.
[Thus, $X_n$ has a negative binomial distribution $\NBi(n,p_0)$.]
Note that $\E X_n = n\E\zeta_1=n(1-p_0)/p_0$, and that $X_n$ is strongly
concentrated about its mean; for example, moment convergence in the central
limit theorem for $X_n$ implies that
%
\begin{equation}
\label{aw3} \P\biggl(\biggl|X_n-\frac{1-p_0}{p_0}n\biggr|>n^{2/3}
\biggr) = O\bigl(n^{-b}\bigr)
\end{equation}
for any fixed $b$. Furthermore,
%
\begin{equation}
\label{aw4} \P\bigl(\xi^*_i=n\bigr)=(1-p_0)^{-1}p_{n+1}=
\bigl(1+o(1)\bigr) (1-p_0)^{-1}\LL(n)n^{-\gb}
\end{equation}
as ${n\to\infty}$, and thus, by a more general version of \eqref{sw}
applied to
$\xi^*_i$
and \eqref{exix},
uniformly for all $k$ with $|k-n(1-p_0)/p_0|\le n^{2/3}$,
%
\begin{eqnarray}
\label{aw5} \P\Biggl(\sum_{i=1}^{k}
\xi^*_i=n-1\Biggr) &=& k\bigl(1+o(1)\bigr)\P \bigl(
\xi^*_1=\bigl\lfloor n-k\E\xi^*_1-1\bigr\rfloor \bigr)
\nonumber\\
& =&\bigl(1+o(1)\bigr)\frac{n(1-p_0)}{p_0} \P \bigl(\xi^*_1=\bigl
\lfloor n(1-\mean )/p_0\bigr\rfloor+o(n) \bigr)
\\
& =&\bigl(1+o(1)\bigr)\frac{n}{p_0} \LL(n) \bigl(n(1-\mean)/p_0
\bigr)^{-\gb}.\nonumber
\end{eqnarray}
Choose $b= \beta+1$. By \eqref{aw1}--\eqref{aw5},
\begin{eqnarray*}
\P\bigl(N^{(0)}=n\bigr) &=& \frac{1}n \P\Biggl(\sum
_{i=1}^{X_n}\xi^*_i=n-1\Biggr) 
\\
& =&\bigl(1+o(1)\bigr){p_0}^{-1}\LL(n) \bigl(n(1-
\mean)/p_0 \bigr)^{-\gb}. 
\end{eqnarray*}
\upqed\end{pf}
\end{appendix}

\section*{Acknowledgements}
S. \"O. Stef\'ansson would like to thank J\'{e}r\'{e}mie Bouttier for telling him about
\cite{legall:2011} and for suggesting to relate the condensation phase in
simply generated trees to planar maps. We would also like to thank
Jakob Bj\"{o}rnberg for discussions and comments.

%

%




\printaddresses

\end{document}